\documentclass[11pt]{article}

\usepackage[T1]{fontenc}
\usepackage[utf8]{inputenc}
\usepackage{lmodern}

\usepackage{amsmath,amssymb,amsfonts}
\usepackage{amsthm}

\usepackage{graphicx}
\usepackage{subfig}      
\usepackage{booktabs}

\usepackage{algorithm}
\usepackage{algpseudocode}

\usepackage{xcolor}
\usepackage[numbers,sort&compress]{natbib}
\usepackage{hyperref}

\usepackage{microtype}

\newcommand{\norm}[1]{\left\lVert#1\right\rVert}
\newcommand{\inner}[1]{\left\langle#1\right\rangle}
\newcommand{\xstar}{x^{\star}}
\DeclareMathOperator*{\argmin}{argmin}
\newcommand{\bbE}{\mathbb{E}}
\newcommand{\R}{\mathbb{R}}

\newtheorem{lemma}{Lemma}
\newtheorem{theorem}{Theorem}
\newtheorem{remark}{Remark}
\newtheorem{assumption}{Assumption}
\newtheorem{example}{Example}

\newcommand{\keywords}[1]{\par\noindent\textbf{Keywords:} #1}

\title{Accelerated optimization algorithms and ordinary differential equations: the convex non Euclidean case}

\author{
Paul Dobson\\
Maxwell Institute for Mathematical Sciences and Mathematics Department,\\
Heriot-Watt University, Edinburgh, EH14 4AS, UK\\
\texttt{p.dobson\_1@hw.ac.uk}
\and
Jes{\'u}s Mar{\'\i}a Sanz-Serna\\
Departamento de Matem\'aticas, Universidad Carlos III de Madrid,\\
Avenida Universidad 30, 28911, Legan\'es, Madrid\\
\texttt{jmsanzserna@gmail.com}
\and
Konstantinos Zygalakis\\
Maxwell Institute for Mathematical Sciences and School of Mathematics,\\
University of Edinburgh, Peter Guthrie Tait Rd, EH9 3FD, Edinburgh\\
\texttt{k.zygalakis@ed.ac.uk}
}

\date{} 

\begin{document}
\maketitle

\begin{abstract}
We study the connections between ordinary differential equations and optimization algorithms in a non-Euclidean setting. We propose a novel accelerated algorithm for minimizing convex functions over a convex constrained set. This algorithm is a natural generalization of Nesterov's accelerated gradient descent method to the non-Euclidean setting and can be interpreted as an additive Runge-Kutta algorithm. The algorithm can also be derived as a numerical discretization of the ODE suggested by Krichene et al.\ in 2015. We use Lyapunov functions to establish convergence rates for the ODE and show that the discretizations considered achieve acceleration beyond the setting studied by Krichene and his coworkers. Finally, we discuss how the proposed algorithm connects to various equations and algorithms in the literature.
\end{abstract}

\keywords{Lyapunov function, probability simplex, convex optimization, mirror map, gradient descent, accelerated methods}

	\section{Introduction}
	Optimization lies at the heart of many problems in statistics and machine learning. We are interested in solving the following problem
	\begin{equation} \label{eq:min}
		\min_{x \in \mathcal{X}}f(x)
	\end{equation}
	where  \(\mathcal{ X}\subseteq\mathbb{R}^d \) is {\color{black}non-empty}, closed and convex and the objective function \(f\) is convex and continuously differentiable in an open set that contains \(\mathcal{X}\). Numerous different algorithms \citep{P87,BV04,NW06,B17} have been proposed for this problem both in the case where $\mathcal{X}=\mathbb{R}^{d}$ as well as when $\mathcal{X}$ is a proper convex subset of $\mathbb{R}^{d}$. These different algorithms can be classified depending on the type of information they use. For example, first-order methods make use only of the gradient of $f$, $\nabla f$. Second-order methods use additionally second derivatives utilizing in some shape or form the Hessian of $f$, $\nabla^{2} f$ and this enables faster convergence.

	In the last few years, first-order methods have gained in popularity despite their slower convergence rate since data sets and problems have become larger \citep{WR22}. In the unconstrained case, the simplest first-order method is gradient descent which however does not give optimal convergence rates within the class of (strongly) convex and gradient Lipschitz functions \citep{N14}. On the other hand, Nesterov proposed a family of accelerated first-order methods   with optimal convergence rates. In the constrained case, one popular algorithm for its simplicity is projected gradient descent \citep{B14}, which is the composition of a gradient descent step and a projection {\color{black} onto}  $\mathcal{X}$. On the other hand mirror gradient  descent \citep{B17}, which is an adaptation of gradient descent with the Euclidean distance replaced by a Bregman divergence, avoids calculating such projection. Nevertheless, both methods fail to yield optimal convergence rates for convex and gradient Lipschitz functions. Two examples of accelerated methods in this constrained setting are the methods proposed in \cite{KBB15} and \cite{T08}.
	
	In a different direction, in the last few years, there has been a renewed interest in connecting optimization algorithms with ordinary differential equations (ODEs). This has been partially driven by the desire to understand better the acceleration phenomenon, starting with the seminal paper of  \cite{SBC16} that showed in the convex case Nesterov's accelerated method corresponds to a discretization of a particular second-order ODE. This result sparked a lot of subsequent research on the links between optimization and numerical solutions of ODEs \citep{Scieur:2016,WRJ21} as well as borrowing ideas from dynamical systems and control theory to prove convergence of optimization algorithms \citep{WRJ21,LRP16,FRMP18}.
	
	An important question  is why certain discretizations of second-order ODEs like the one that appeared in \cite{SBC16}  may or may not accelerate.  The papers  \citep{SDJ18,SDS19}  explain the behaviour of different algorithms using the high-resolution ODEs framework. {\color{black} In particular, for each optimization algorithm an ODE is derived that follows the dynamics of the corresponding algorithm closely and then it is analysed to highlight the difference between the different optimization algorithms.}
	
	Following a different direction \cite{SSKZ21}  {\color{black} highlighted the importance of \emph{stability} {\color{black} when} constructing numerical discretizations that yield acceleration. More precisely, using ideas from integral quadratic constraints \citep{LRP16, FRMP18}, a Lyapunov function was constructed for a family of numerical discretizations of an appropriate second order ODE. Due to consistency, the corresponding Lyapunov functions of the different numerical discretization converge to the Lyapunov function of that ODE as the learning rate tends to zero. However,  in order for a numerical discretization to be able to achieve acceleration, consistency is not sufficient and there is a stability issue:  one needs to be able to construct a discrete Lyapunov function for suitably large learning rates, which is not the case for the majority of the numerical discretizations. A set of conditions were derived   that a numerical discretization should satisfy to yield acceleration.}  Furthermore,  these ideas are extended in \cite{DSZ24} to obtain sharper convergence rates for the Nesterov's accelerated method in the strongly convex case as well as giving an interpretation of it as an additive Runge-Kutta {\color{black} numerical integrator} \citep{Cooper80,Cooper83}.


{\color{black}
The papers we have just discussed, relating ODEs, their numerical discretizations and optimisation algorithms, only consider the Euclidean setting.	
  An early important reference  that addresses these issues in non-Euclidean contexts is \cite{KBB15}. This paper generalises the ODE from \cite{SBC16}  using appropriate Bregman divergences and  proposes a discretization  which uses  a suitable regularizing function and inherits the favourable acceleration properties of the ODE. An alternative discretisation of this ODE based on variable splitting was very recently proposed in \cite{Mirror26}, while  \cite{YXZ25} explained the  convergence properties of the algorithm in \cite{KBB15}  using the framework of high resolution ODEs  \cite{SDJ18,SDS19}. An alternative way to generalise the ODE from \cite{SBC16} is given in
\cite{WRJ21}, which is analysed using Lagrangians that contain appropriate Bregman divergences.}
	
	
	In this work, we provide a novel discretization of the ODE appearing in \cite{KBB15} leading to an accelerated optimization method for the problem \eqref{eq:min} without requiring the use of a regularisation function. The proposed method is the natural extension of Nesterov's method in the non-Euclidean setting and corresponds to an additive Runge Kutta discretization of the underlying non-Euclidean ODE. Furthermore, we extend the analysis from \cite{KBB15} both for the ODE dynamics as well as their discretizations to be able to deal with scenarios not previously covered, for example in the case of simplex when the minimizer is on the boundary. Finally, we make explicit connections between the various non-Euclidean ODEs \citep{KBB15,WRJ21} and optimization algorithms \citep{T08}. 	
	
	The rest of the paper is organised as follows. In Section \ref{sec:Euclidean_setting} we revisit the problem when $\mathcal{X}=\mathbb{R}^{d}$ {\color{black}and the setting is Euclidean and recall some aspects of} the connection between ODEs and optimization algorithms. We then in Section \ref{sec:mirror_sec}  discuss in detail a natural generalization of gradient flow and gradient descent in the non-Euclidean setting.
	Section \ref{sec:non-Euclidean} contains our main results. In particular, after summarizing the results in \cite{KBB15}, we propose our novel algorithm and explain why it generalizes Nesterov's method in the non-Euclidean setting as well as its connection with additive Runge-Kutta methods. Furthermore, we establish convergence for our algorithm in the setting proposed in \cite{KBB15} and extend these results in a more general setting. Several numerical experiments are presented in {\color{black} Section} \ref{sec:numerics} that illustrate the behaviour of our proposed method as well as compare it with the method in \cite{KBB15}. {\color{black}There is an appendix devoted to discussing the relations between the ODEs and the algorithms in this paper and counterparts, previously considered in the literature, that use only primal variables.}

	\section{ODEs and optimization in the Euclidean setting}\label{sec:Euclidean_setting}
	In this preparatory section we consider the particular case where in \eqref{eq:min} \(\mathcal{X}=\R^d\) and \(\R^d\) is endowed with the standard Euclidean norm.
	
	\subsection{Gradient flow and gradient descent}
	The simplest ODE associated with the problem \eqref{eq:min} is given by the gradient flow
	\begin{equation} \label{eq:gf}
		\dot{x}(t)=-\nabla f(x(t)).
	\end{equation}
	When \eqref{eq:min} has a minimizer \(x^\star\), \(f(x(t))-f(x^\star)\) approaches \(0\) at a rate \(\mathcal{O}(1/t)\), as it may be proved e.g. by means of the Lyapunov function
	\begin{equation} \label{eq:Lyap1}
		V(x,t)= t(f(x)-f(\xstar))+\frac{1}{2} \norm{x-\xstar}^{2}.
	\end{equation}
	In fact, since \(V\) is nonincreasing along solutions of \eqref{eq:gf}, one has
	\[
	f(x(t))-f(\xstar) \leq \frac{1}{t} V\big(x(0),0\big)= \frac{1}{2t}\norm{x(0)-\xstar}^{2}, \qquad t >0.
	\]

	The simplest method to discretize \eqref{eq:gf} is the explicit Euler rule, that leads to the standard gradient descent algorithm
	\begin{equation} \label{eq:gd}
		x_{k+1} = x_{k}-h \nabla f(x_{k}),
	\end{equation}
	where $h$ is the timestep/learning rate.
	Since $x_{k}$ {\color{black}is an approximation to} $x(kh)$  (the solution of the differential equation at time $t=kh$), {\color{black} it is plausible to expect} from the continuous time analysis that $f(x_{k})-f(\xstar)$ would decay like $1/k$. In fact, when \(f\) is \(L_f\)-smooth, i.e.
	\[
	\forall x,y\in\R^d, \qquad  \norm{ \nabla f(x)-\nabla f(y)} \leq L_f \norm{x-y},
	\]
	and \(h\leq 1/L_f\), the \(\mathcal{O}(1/k)\) decay may be proved
	by using the discrete Lyapunov function
	\[
	V_{k}(x)= k h(f(x)-f(\xstar))+\frac{1}{2} \norm{x-\xstar}^{2},
	\]
	which is the obvious counterpart of \eqref{eq:Lyap1}.  A more sophisticated Lyapunov function valid for \(1/L_f\leq h< 2/L_f\) may be seen in \cite{FRMP18}.
	
	There is no need to derive \eqref{eq:gd} by seeing it as an integrator for an ODE. The algorithm may be written as
	\begin{equation} \label{eq:maxmin}
		x_{k+1}=\argmin_{x\in\mathbb{R}^d}  \left \{f(x_{k})+\inner{\nabla f(x_{k}),x-x_{k}}+\frac{1}{2h} \norm{x-x_{k}}^{2} \right \},
	\end{equation}
	with a clear optimization interpretation: \(f(x_{k})+\inner{\nabla f(x_{k}),x-x_{k}}\) is the linear approximation to the objective function and \(({1}/(2h)) \norm{x-x_{k}}^{2}\) {\color{black} may be understood as a penalty term and also as the best quadratic approximation to \(f\) that does not require computing its Hessian}.
	
	\subsection{Acceleration}
	Algorithms with acceleration for convex objective functions, including the celebrated Nesterov method \citep{N14}, raise to \(\mathcal{O}(1/k^2)\) the \(\mathcal{O}(1/k)\) rate of convergence of gradient descent.
	Even though no system of ODEs was originally used to derive  Nesterov's algorithm, the well-known reference
	\cite{SBC16} showed the relation between the algorithm and the second-order ODE
	\[
	\ddot{x}(t)+\frac{r+1}{t}\dot{x}(t)+\nabla f(x(t))=0,
	\]
	where \(r\geq 2\).
	For our purposes, it is useful to rewrite the equation as a first order system
	\begin{subequations} \label{eq:Polyak_first}
		\begin{eqnarray}
			\dot{z}(t) &=& - \frac{t}{r} \nabla f(x(t)), \\
			\dot{x}(t) &=& \frac{r}{t}\left(z(t)-x(t) \right),
		\end{eqnarray}
	\end{subequations}
	with Lyapunov function
	\[
	V(x,z,t)=\frac{t^{2}}{r^{2}}(f(x)-f(\xstar))+\frac{1}{2}\norm{z-\xstar}^{2},
	\]
	which implies the following \(\mathcal{O}(1/t^2)\) decay estimate of $f(x(t))$ towards the optimal value $f(\xstar)$
	\[
	f(x(t))-f(\xstar) \leq \frac{r^2}{t^2} V\big(x(0),z(0),0\big)= \frac{r^2}{2t^2}
	\norm{z(0)-\xstar}^2, \qquad t>0.
	\]
	
	Following the line of thought in the preceding subsection, one would expect that integrators $(z_{k},x_{k}) \mapsto (z_{k+1},x_{k+1})$ for \eqref{eq:Polyak_first},
	where $(z_{k},x_{k})$ approximate $(z(k\delta), x(k\delta))$ ($\delta$ is the time-step), may offer the potential of yielding optimization algorithms for which $f(x_{k})-f(\xstar)$ decays like $1/k^{2}$, i.e.\ algorithms that show acceleration. Unlike the case of the gradient flow, this is not a simple task since standard discretizations such a Runge-Kutta algorithms do not lead to acceleration, see the discussion in \cite{SSKZ21} and \cite{DSZ24}. In addition, even for discretizations with acceleration, a Lyapunov function of the ODE, may not work for the discrete algorithm.
	
	
	As proved in a more general setting in Section~\ref{subsec:alg_new}, if
	\(\{\gamma_k\}_{k=0}^\infty\) is a sequence with \(\gamma_0=1\), \(\gamma_k \geq 1\),  \(k= 1, 2,\dots\), the algorithm
	\begin{subequations}\label{eq:nest_euclidean}
		\begin{eqnarray}
			\label{eq:algyeuc}
			y_k &=& x_k +\frac{1}{\gamma_k} (z_k-x_k),\\
			\label{eq:algzeuc}
			z_{k+1}&=& z_k -\gamma_k h\nabla f(y_k),\\
			\label{eq:algxeuc}
			x_{k+1} &=& y_k +\frac{1}{\gamma_k}(z_{k+1}-z_k),
		\end{eqnarray}
	\end{subequations}
	provides, for suitable choices of the learning rate \(h\) and  the coefficients \(\gamma_k\), a \emph{consistent}, albeit nonstandard, discretization of \eqref{eq:Polyak_first} that leads to acceleration. {\color{black} (The standard definition of \emph{consistency}  may be seen in most numerical analysis text-books, see e.g. \cite[Definition 12.1]{suli}.)}
	\begin{figure}[t!]
		\centering
		\includegraphics[width=0.6\textwidth]{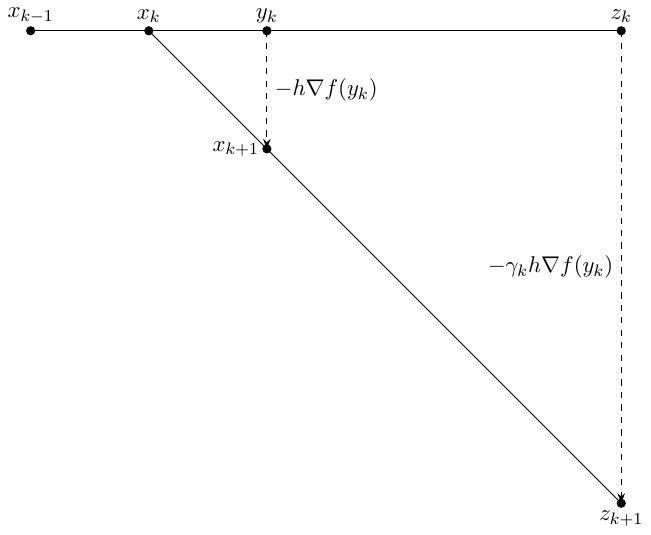}
		\caption{An illustration of one step of the Nesterov algorithm in a Euclidean setting}
		\label{fig:Nesterov_diag}
	\end{figure}
	Figure \ref{fig:Nesterov_diag} illustrates one step of \eqref{eq:nest_euclidean}. The point $y_k$ is determined as a convex combination of $x_k$ and  $z_k$,  then $z_{k+1}$ is obtained by moving from $z_k$ in the direction of the gradient and finally $x_{k+1}$ is obtained as a convex combination of $x_k$ and $z_{k+1}$.
	
	The algorithm may be reformulated by eliminating the variable \(z\). Using   \eqref{eq:algzeuc} and \eqref{eq:algxeuc}, we obtain
	\begin{equation}\label{eq:euclideo1}
		x_{k+1} = y_k-h\nabla f(y_k)
	\end{equation}
	and, after setting \(z_k = x_k+(\gamma_{k-1}-1)(x_k-x_{k-1})\),  \eqref{eq:algyeuc} becomes, for \(k\geq 1\),
	\begin{equation}\label{eq:euclideo2}
		y_k = x_k+\beta_{k-1}(x_k-x_{k-1}),\qquad \beta_{k-1} = (\gamma_{k-1}-1)/\gamma_k,
	\end{equation}
	(\(y_0=x_0\)).
	In the formulation \eqref{eq:euclideo1}--\eqref{eq:euclideo2}, one first computes $y_k$ by extrapolation from $x_{k-1}$ and $x_k$ and then moves from $y_k$ to $x_{k+1}$ by a gradient descent substep.
	The relations \eqref{eq:euclideo1}--\eqref{eq:euclideo2} are the well-known formulas for the accelerated Nesterov method
	\citep{N14} when written as a three-term recursion linking \(x_{k-1}\), \(x_k\) and \(x_{k+1}\) as in e.g.\ \cite{FRMP18}, with
	\begin{equation}\label{eq:fancygamma}
		\gamma_k =\frac{1}{2} \big(1+\sqrt{1+4\gamma_{k-1}^2}\big ),\quad k=1,2,\dots, \qquad\gamma_0 = 1.
	\end{equation}

	\section{Non-Euclidean optimization}\label{sec:mirror_sec}
	It is well known that, in many instances (see e.g.\ the discussion in \cite[Example 9.19]{B17}), it is useful to consider the problem \eqref{eq:min} when the norm in \(\R^d\) is not Euclidean.  In what follows, \( \bbE\) denotes \(\R^d\) endowed with an {\color{black}\emph{arbitrary}} norm \(\|\cdot\|\) and \( \bbE^\star\) is the dual space  with dual norm  \(\|\cdot\|_\star\). The map \(\langle \cdot,\cdot\rangle:\bbE^\star\times\bbE\rightarrow \R\) denotes the standard   pairing between \(\bbE^\star\) and \(\bbE\), i.e.\ the real number \(\langle \zeta,  x\rangle=\sum_j\zeta_jx_j\) is the value of the linear form \(\zeta \in\bbE^\star\) acting on the vector \(x\in\bbE\).
	
	\subsection{Mirror gradient ODE and mirror descent}\label{sec:mirror}
	In the non-Euclidean setting, the gradient flow equation \eqref{eq:gf} is meaningless because \(\dot x(t)\in \bbE\) is a primal vector and \(\nabla f(x(t))\in\bbE^\star\) is a dual vector.  Nemirovsky and Yudin \citep{NY83} suggested alternative ODEs of the form
	\begin{equation} \label{eq:mirror_gf1}
		\dot{\zeta}(t)=-\nabla f(\chi(\zeta(t))), \qquad x(t)=\chi(\zeta(t)),
	\end{equation}
	where \(\zeta\) takes values in \(\bbE^\star\) and \(\chi\) maps \(\bbE^\star\) into \(\bbE\). The dynamics take place in the dual space and the primal variable \(x\) just \emph{mirrors} the behaviour of \(\zeta\); for this reason, \(\chi\) is referred to as the \emph{mirror map}. In \cite{NY83}, it is assumed that \(\chi=\nabla \psi^\star\), where \(\psi^\star:\bbE^\star\rightarrow \R\) is a differentiable function
	which is used to construct Lyapunov functions for \eqref{eq:mirror_gf1}. There is much freedom  in the choice of \(\psi^\star\) \citep{NY83}; throughout this paper the attention is restricted to cases where the following standing requirement is met.
	\begin{assumption} \label{as:cond_mir1}
		The  mirror map satisfies $\chi=\nabla \psi^{\star}$ for some convex and differentiable function $\psi^{\star}: \bbE^{\star} \rightarrow \mathbb{R}$  and takes values in $\mathcal{X}$. In addition, {\color{black} \(\chi\)} is \(L_\chi\)-{\color{black} Lipschitz continuous} i.e.\ for each \(\xi,\zeta\in\bbE^\star\):       %
		\[ \|\chi(\xi)- \chi(\zeta)\| \leq L_\chi \|\xi-\zeta\|_\star.\]
	\end{assumption}
	
	If \(D_{\psi^\star}\) denotes the corresponding Bregman divergence, so that for \(\xi,\zeta\in\bbE^\star\),
	\[
	D_{\psi^\star}(\xi,\zeta) =  \psi^\star(\xi)-\psi^\star(\zeta)-\langle \xi-\zeta, \nabla \psi^\star(\zeta) \rangle,
	\]
	and \(\zeta^\star\) is such that \(x^\star = \chi(\zeta^\star)\) is a minimizer of \eqref{eq:min}, it is easy to prove (see e.g.\ \cite{KBB15}) that
	\begin{equation}\label{eq:dstar}
		\frac{d}{dt} D_{\psi^\star}(\zeta(t),\zeta^\star)\leq -\big(f(x(t))-f(x^\star)\big)\leq 0,
	\end{equation}
	a fact that may be used to establish convergence \citep{NY83}.
	
	Some {\color{black} important settings where Assumption \ref{as:cond_mir1} holds are now described}.
	\begin{itemize}
		\item \emph{Unconstrained Euclidean case.} Here \({\cal X} = \bbE\), the norms \(\|\cdot\|\) and \(\|\cdot\|_ \star\) are Euclidean, and \(\psi^\star(\cdot) = (1/2)\|\cdot\|^2 \). In this case, the spaces \(\bbE\) and \(\bbE^\star\) may be identified with one another, \(\chi\) is the identity, \(D_{\psi^\star}(\xi,\zeta) = (1/2)\|\xi-\zeta\|^2\) for each \(\xi\) and \(\zeta\), and the pairing \(\langle\cdot,\cdot\rangle\) may be identified with the Euclidean inner product.
		The ODE \eqref{eq:mirror_gf1} reduces to the gradient flow equation \eqref{eq:gf}
		\item \emph{The simplex.} Here \(\cal{X}\) is the probability simplex
		\[
		\Delta = \Big\{ x=(x_1,\dots,x_d)\in\R^d: \sum_{j=1}^d x_j = 1,\:\:\: x_j \geq 0,\: j=1,\dots,d\Big\}.
		\]
		Choosing
		\(\psi^\star(\zeta) = \log \sum_j e^{\zeta_j}\),
		the \(i\)-component of \(\chi(\zeta)\) equals  \( e^{\zeta_i}/\sum_j e^{\zeta_j}\).
		The mirror map is \(1\)-smooth  \cite[Example 5.15]{B17} when \(\|\cdot\|_ \star\) is either the \(\ell^\infty\) or the \(\ell^2\) norm (in which case \(\|\cdot\|\) is respectively the \(\ell^1\) or the \(\ell^2\) norm).
		\item \emph{The hypercube.} \({\cal X}= [0,1]^d\). One may choose
		\(\psi^\star(\zeta)=\sum_j \log(e^{\zeta_j}+1)\),
		so that the \(i\)-th component of  \(\chi(\zeta)\) is \(e^{\zeta_i}/(e^{\zeta_i}+1)\). Now \(\chi\) is a diffeomorphism of \(\bbE^\star\) onto the interior
		\((0,1)^d\) of \(\cal X\). If \(\|\cdot\|\) is any of the \(\ell^p\)    norms, \(p\in[1,\infty]\),  it is easy to check that \(L_\chi = 1/4\).
	\end{itemize}
	
	The Euler discretization of \eqref{eq:mirror_gf1} yields the following mirror descent algorithm
	\begin{equation}\label{eq:descentdual}
		\zeta_{k+1} = \zeta_k -h \nabla f(\chi(\zeta_k)),\qquad x_{k+1} = \chi(\zeta_{k+1}).
	\end{equation}
	As in the ODE, the primal variable \(x\) just mirrors the evolution of the dual variable \(\zeta\).
	
	\subsection{Writing the mirror flow ODE and mirror descent in the primal space}\label{ss:rewritingprimal}
	
	We now write the ODE \eqref{eq:mirror_gf1} in terms of the primal variable \(x\). If \(\chi\) is differentiable with Jacobian $\chi'$, \eqref{eq:mirror_gf1} implies:
	\begin{equation}\label{eq:odeaux}
		\dot{x}(t) =-\chi^\prime(\zeta(t))\nabla f(x(t)).
	\end{equation}
	This equation  still contains \(\zeta\);
	in order to eliminate \(\zeta\), we have to demand that the mirror map, that we recall it is assumed throughout to satisfy Assumption \ref{as:cond_mir1}, has additional properties {\color{black}listed in Assumption \ref{ass:two} below. We first present some notation.

	The symbols \(\cal A\), \(\cal V\) will denote respectively the \emph{affine hull of \(\cal X\) and the corresponding linear subspace}:
	\begin{align*}
		\mathcal{V}  &= \mathrm{span}\{x-z:x,z\in\mathcal{X}\},\\
		\mathcal{A} &= \mathcal{X}+\mathcal{V}.
	\end{align*}
	For the simplex, \(\cal A\) has the equation \(\sum_j x_j= 1\) and \(\cal V\) consists of all vectors with \(\sum_j x_j = 0\). For Euclidean case and the hypercube \(\cal A\) and \(\cal V\) are the whole space \(\bbE\).

\begin{table}\color{black}
\begin{tabular}{c|c|c|c}
&Euclidean & Simplex & Hypercube\\\hline
$\cal X$ & $\R^d$ & $\sum_j x_j = 1$, $x_j \geq 0$, \(\forall j\)& $0\leq x_j\leq 1$, \(\forall j\)\\
ri($\cal X$)& $\R^d$ & $\sum_j x_j = 1$, $x_j > 0$, \(\forall j\)& $0 < x_j<1$, \(\forall j\)
\\
$\phi$ & $(1/2) \|\cdot\|^2$
& $ \begin{cases}\sum_j x_j\log x_j, & {\rm if\ }  x_j\geq 0, \forall j\\ \infty & {\rm else}\end{cases}$
& $\begin{cases}\sum_j \big(x_j\log x_j+(1-x_j)\log(1-x_j)\big), &{\rm if\ }  0\leq x_j\leq 1, \forall j\\ \infty &{\rm else}\end{cases}$
\\
$\psi^\star$ & $(1/2) \|\cdot\|^2$ & $\log \sum_j e^{\zeta_j}$ & $\sum_l \log(e^{\zeta_j}+1)$
\\
$\chi_i$& $\zeta_i$ & $e^{\zeta_i}/\sum_j e^{\zeta_j}$ & $e^{\zeta_i}/\sum_j (e^{\zeta_j}+1)$
\end{tabular}
\caption{\color{black} Summary of some possible settings. The notation $\chi_i$ refers to the $i$-th component of the mirror map $ \chi = \nabla \psi^\star$, with
$\psi^\star$ the convex conjugate of \(\phi+\delta_{\cal X}\) }
\label{table}
\end{table}

The notation \({\cal N}\) refers to the vector subspace of \(\bbE^\star\) \emph{orthogonal} to \({\cal V}\subseteq \bbE\); thus \(\zeta\in \cal N\) if and only if \(\langle \zeta, x\rangle= 0\) for each \(x\in\cal V\), or, equivalently, if and only if \(\langle \zeta, x_1-x_2\rangle= 0\) for all \(x_1\), \(x_2\) in \(\cal A\). For the simplex, \(\cal N\) is spanned by the vector \(\bf 1\) whose entries are all equal to \(1\) (see the left panel in Figure~\ref{fig:simplex_diagram}). For the Euclidean case and the hypercube, \({\cal V} = \bbE\) and therefore \({\cal N} = \{0\}\).

The notation  \({\rm ri}({\cal X})\) stands for the \emph{relative interior} of \(\cal X\), i.e. for the interior of  \(x\in{\cal X}\) when viewed as a subset of the topological space \(\cal A\); this is always nonempty as \(\cal X\) is nonempty and convex \citep{B17}.
 For the simplex, the relative interior is the
the subset \(\Delta_+\) of \(\Delta\) consisting of points with  positive components. For the hypercube, the relative interior consists of the points  with components \(0<x_i<1\).

As it is standard, \(\delta_{\cal X}\) denotes the \emph{indicator function} that vanishes in \(\cal X\) and takes the value \(\infty\) elsewhere and the \emph{convex conjugate} of  a function \(F:\bbE\rightarrow [-\infty,\infty]\)  is the function \(F^\star:\bbE^\star\rightarrow [-\infty,\infty]\) defined by:
\begin{equation*}
F^\star(\zeta) = {\rm max}\{\langle \zeta,x\rangle-F(x): x\in\bbE\}\qquad \zeta\in\bbE^\star.
\end{equation*}

We are now prepared to present the following additional assumption, that will be required ---in addition to Assumption \ref{as:cond_mir1}, which is assumed throughout--- in \emph{some} of the results to be presented:
}	
	\begin{assumption}\label{ass:two}The mirror map \(\chi=\nabla \psi^\star\) is differentiable and its
		image  is the relative interior \({\rm ri}({\cal X})\). In addition, \(\psi^\star\) is the convex conjugate of a function \(\psi=\phi+\delta_{\cal X}\), where \(\phi:\bbE\rightarrow (-\infty,\infty]\) is  proper, convex,  differentiable over \({\rm ri}({\cal X})\) and satisfies \({\cal X}\subseteq {\rm dom}(\phi)\).
	\end{assumption}
	
This additional assumption holds in the three {\color{black} particular instances previously}  considered. In the Euclidean setting, \(\phi(\cdot) = (1/2)\|\cdot\|^2\). For the simplex, \(\phi\) is given by\begin{equation}\label{eq:negent}
		\phi(x) = \sum_{j=1}^n x_i\log x_i
	\end{equation}
	(negative entropy) if \(x\) is in the nonnegative orthant, with \(\phi(x)= \infty\) else. For the hypercube,
	\[
	\phi(x) = \sum_{j=1}^n \big(x_i\log x_i+(1-x_i)\log(1-x_i)\big),
	\]
	(negative bit entropy) if \(x\in[0,1]^d\), with \(\phi(x)= \infty\) else. See Table~\ref{table} for a summary.

	We next present a lemma that we shall use repeatedly.
	
	\begin{figure}
		\centering
		{\includegraphics[width=0.45\textwidth]{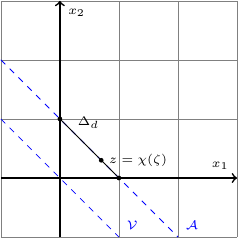}}\hfill
		{\includegraphics[width=0.45\textwidth]{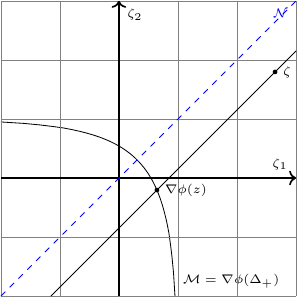}}\hfill
		\caption{Lemma~\ref{lemma} for the case of the simplex. The left and right panels correspond to the primal and dual spaces. Each point $\zeta\in \bbE^\star$ is mapped by $\chi$ into a point $z=\chi(\zeta)\in{\rm ri}({\cal X})=\Delta_{+}$; the image $\nabla \phi(z)$ does not in general coincide with $\zeta$, but $\zeta$ and $\nabla \phi(z)$ differ in an element in $\cal N$. The straight lines of slope 1 partition the dual space; each line is mapped into a single point by $\chi$.}
		\label{fig:simplex_diagram}
	\end{figure}
	
	\begin{lemma}\label{lemma}
		If Assumptions \ref{as:cond_mir1} and \ref{ass:two} hold, then:
		\begin{enumerate}
			\item As \(z\) varies in \({\rm ri}({\cal X})\), the inverse images by the mirror map \(\chi^{-1}(z)\)\( = \{\zeta\in\bbE^\star:\chi(\zeta) = z\}\) provide a partition of  \(\bbE^\star\) into pairwise disjoint sets.
			\item For each \(z\in{\rm ri}({\cal X})\),  its inverse image \(\chi^{-1}(z)\) is the affine set  \(\nabla \phi(z)+{\cal N}\). In particular, \( \nabla \phi(z) \in \chi^{-1}(z)\), i.e.\ \(\chi(\nabla \phi(z)) = z\).
			\item If \(\zeta\in\bbE^\star\), then \(\nabla\phi(\chi(\zeta))-\zeta \in \cal N\). Therefore
			for each \(x_1\), \(x_2\) in the affine hull \(\cal A\) of \(\cal X\),{\color{black}
			\begin{equation}\label{eq:partthree}\langle \nabla\phi(\chi(\zeta))-\zeta, x_1-x_2\rangle = 0.\end{equation}}
			\item If  \(\zeta\in\bbE^\star\) and \(\eta\in\cal N\), then{\color{black}
			\begin{equation}\label{eq:partfour}
			\chi(\zeta) = \chi(\zeta+\eta),\qquad \chi^\prime(\zeta) = \chi^\prime(\zeta+\eta).
			\end{equation}}
		\end{enumerate}
	\end{lemma}
	\begin{proof}
		
	\emph{Part 1.} This is a trivial consequence of the fact that \(\chi\) maps \(\mathbb{E}^\star\) onto \({\rm ri}({\cal X})\).
	
	\noindent\emph{Part 2.} Fix \(z\in {\rm ri}({\cal X})\). Since \(\chi = \nabla \psi^\star\), by a well-known result on conjugate functions, see e.g.\ \citep[Theorem 4.20]{B17}, the relation \(\chi(\zeta) = z\) is equivalent to \(\zeta\in \partial \psi(z)\). It is then sufficient to prove that
	\(\partial \psi(z)=\nabla \phi(z)+{\cal N} \).
	By definition, \(g\in\partial \psi(z)\) means that, for each \(x\in\mathbb{E}\),
	\(
	\psi(x) \geq \psi(z) +\langle g, x-z\rangle
	\). This inequality trivially holds if \(x\notin \cal X\), because then \(\psi(x) = \phi(x)+\delta_{\cal X}(x)=\infty\). Therefore \(g\in\partial \psi(z)\) if and only if, for \(x\in\cal X\), \(
	\psi(x) \geq \psi(z) +\langle g, x-z\rangle
	\), that is \(
	\phi(x) \geq \phi(z) +\langle g, x-z\rangle
	\). On the other hand, since \(\phi\) is differentiable at \(z\), the linear approximation at \(z\) based on the use of \(\nabla \phi(z)\) given by \(\ell(x) = \phi(z) +\langle g, x-z\rangle\) is unique in satisfying \(\phi(x) \geq \ell(x)\) for each \(x\in\mathbb{E}\). In this way
	\(g\in\partial \psi(z)\) if and only if
	\( \langle g, x-z\rangle = \langle \nabla\phi(z),x-z\rangle\)
	for each \(x\in\cal X\), which in turn is clearly equivalent to \(g-\nabla\phi(z)\in\cal N\).
	
	\noindent\emph{Part 3.} We have \(\zeta\in \chi^{-1}(\chi(\zeta))\), so that, by Part 2, \(\zeta\in\nabla\phi(\chi(\zeta))+{\cal N}\) or
	\(\nabla\phi(\chi(\zeta))-\zeta \in\cal N\).
	
	\noindent\emph{Part 4.} By Part 2,  \(\chi^{-1}(\chi(\zeta))\) is an affine set associated to the vector space \(\cal N\). Therefore the relations \(\zeta\in\chi^{-1}(\chi(\zeta))\) and \(\eta\in \cal N\) imply \(\zeta+\eta \in \chi^{-1}(\chi(\zeta))\), so that \(\chi(\zeta)=\chi(\zeta+\eta)\). Differentiation leads to \(\chi^\prime(\zeta)=\chi^\prime(\zeta+\eta)\) .
	\end{proof}

	The lemma shows that there is a bijection  between points \(z \in {\rm ri}({\cal X})\) and sets  \(\chi^{-1}(z) = \nabla \phi(z)+{\cal N}\subseteq \bbE^\star\).
	There are two essentially different scenarios:
	\begin{enumerate}
		\item \({\cal N} = \{0\}\). Each set \(\nabla \phi(z)+{\cal N}\) consists of the single point \(\nabla\phi(z)\). In other words,
		the mapping \(\nabla \phi\) restricted to  \({\rm ri}({\cal X})\) is the inverse of \(\chi\) and there is a one-to-one correspondence \(z=\chi(\zeta)\), \(\zeta = \nabla\phi(z)\), between points \(z\in {\rm ri}({\cal X})\) and points \(\zeta\in\bbE^\star\).
		\item \({\cal N} \neq \{0\}\). In this case, for each \(z\in{\rm ri}({\cal X})\), the affine set \(\nabla\phi(z)+\cal N\) has dimension \({\rm dim}({\cal N})>0\). The point \(\nabla \phi(z)\) is in the set \(\nabla\phi(z)+\cal N\) but it does not belong to any set \(\nabla\phi(z^\prime)+\cal N\) if \(z^\prime\neq z\).
		The mapping \(z\in {\rm ri}({\cal X}) \rightarrow\nabla \phi(z)\) provides a one-to-one {\color{black} \emph{parameterization of a manifold \(\cal M\)}}
		of dimension \({\rm dim}(\bbE^\star)-{\rm dim}({\cal N})\). See Figure~\ref{fig:simplex_diagram} for the simplex with $d=2$, where
		\(\cal M\) is a curve.
	\end{enumerate}

	We may now rewrite \eqref{eq:mirror_gf1} in the primal space when the additional Assumption \ref{ass:two} holds.
	We consider two situations.
	\begin{itemize}
		\item \emph{The case \({\cal N} =\{0\}\).} The one-to-one correspondence between \(x=\chi(\zeta)\) and \(\zeta=\nabla \phi(x)\) may be used to
		express
		\eqref{eq:mirror_gf1} as
		\begin{equation}\label{eq:odeaux2}
			\frac{d}{dt} \nabla \phi(x) = -\nabla f(x).
		\end{equation}
		
		When \(\cal N\) is not reduced to \(\{0\}\), this differential equation is meaningless because the left hand-side is constrained to be a vector tangent to the manifold \(\cal M\) and \(\nabla f(x)\) is not constrained in that way (refer to the right panel in Figure~\ref{fig:simplex_diagram}).
		
		\item \emph{General \({\cal N}\).} {\color{black} From \eqref{eq:partthree},}  \(\eta:= \nabla \phi(z)-\zeta\in\cal N\) {\color{black} if $z = \chi(\zeta)$}.
		Then, by {\color{black} using \eqref{eq:partfour}}, \(\chi^\prime(\zeta) = \chi^\prime(\nabla \phi(z))\) and from \eqref{eq:odeaux}
		\begin{equation}\label{eq:odeaux3}
			\dot{ x} =-\chi^\prime(\nabla \phi(x))\nabla f(x).
		\end{equation}
	\end{itemize}
	{\color{black} This ODE has appeared before in the literature in \cite{KBB16}}. In the particular case with \({\cal N} = \{0\}\), this reduces to \eqref{eq:odeaux2}, because \(\chi^\prime(\nabla \phi(x))\) is the inverse
	of \((\nabla \phi(x))^\prime\), as it is seen by differentiating \(\chi(\nabla \phi(x)) = x\).
	
	\begin{example}\label{ex:simplex} In the case of the simplex, straightforward differentiation of the expression for \(\chi\) shows that \eqref{eq:odeaux3} reads:
		\begin{equation}\label{eq:edosimplex}
			\dot{ x}= D(x) \Big( -\nabla f(x)+ \langle -\nabla f(x) ,x \rangle {\bf 1}\Big),
		\end{equation}
		where \(D(x)\) is the diagonal matrix with entries \(x_i\).  {\color{black} This  ODE  plays an important role in evolutionary game theory \cite{W97} and viability theory \cite{JP90} as it can be used to model replicator dynamics.}
		
		For \(x\) in the relative interior, $D(x)$ is \emph{the inverse Jacobian} of the map \(\nabla \phi\), whose components are \(1+\log x_i\) (i.e.\ the inverse of the Hessian of \(\phi\)). The  right-hand side of \eqref{eq:edosimplex} is of the form \(D(x)v\), where the
		vector  \(v = \nabla f(x)+\langle  \nabla f(x) ,x \rangle {\bf 1}\) is a linear combination of  \(\nabla f(x)\) and \(\bf 1\), with a coefficient \(\langle  \nabla f(x) ,x \rangle\) that ensures that \(v\) is tangent  at \(\nabla \phi(x)\) to the manifold \(\cal M\). Multiplying \(v\) by the inverse Jacobian results in a vector \(D(x)v\) that is tangent to the relative interior of the simplex, the inverse image of \(\cal M\) by \(\nabla \phi\). Analytically this corresponds to the fact that, in \eqref{eq:edosimplex},
		\((d/dt) \sum_j x_j = 0\) so that this differential equation preserves the constraint \(\sum_j x_j=1\). Furthermore, due to the  presence of \(D(x)\), the \(i\)-th component of the right hand-side of \eqref{eq:edosimplex} vanishes if \(x_i=0\). Therefore the points on the hyperplanes \(x_i=0\) in \(\bbE\) are equilibria of the differential equation, so that, as \(t\) varies, the \(x_i(t)\) will remain positive if \(x(0)\) is in the relative interior.
	\end{example}

	Similarly, when Assumption \ref{ass:two} holds, it is possible to rewrite the discretization \eqref{eq:descentdual} purely in terms of primal variables. For the case \({\cal N} =\{0\}\), we have
	\begin{equation}\label{eq:threenablas}
		\nabla \phi(x_{k+1})=\nabla \phi(x_k)-h\nabla f(x_k),
	\end{equation}
	that coincides with the Euler discretization of \eqref{eq:odeaux2}. Once \(\nabla \phi(x_{k+1})\) has been computed by this formula, \(x_{k+1 }\) is retrieved as \(\chi(\nabla \phi(x_{k+1}))\).

	For general, \(\mathcal N\) we note that
	\[
	x_{k+1} = \chi\big(\zeta_k -h \nabla f(x_k)\big);
	\]
	{\color{black}From \eqref{eq:partthree}, \(\eta = \nabla \phi(x_k)-\zeta_k\in{\cal N}\), and then, by \eqref{eq:partfour},} we have the following well-known formulation (see \cite[Remark 9.6]{B17}) of the mirror descent algorithm:
	\begin{equation}\label{eq:mirrordescent}
		x_{k+1} = \chi\big(\nabla \phi(x_k)-h\nabla f(x_k)\big).
	\end{equation}
	As it may have been expected, this is a consistent discretization of the differential equation \eqref{eq:odeaux3}, as it may be seen by Taylor expansion of the right hand-side. When \({\cal N} =\{0\}\), the application of \(\nabla \phi\) to \eqref{eq:mirrordescent}  yields \eqref{eq:threenablas}.
	
	The algorithm~\eqref{eq:mirrordescent} may be formulated from an optimization point of view, without using ODEs as stepping stones. Such a formulation is based on the
	the (primal) Bregman divergence \(D_\phi\) associated with \(\phi\):
	\[D_\phi(x,z) = \phi(x)-\phi(z) -\langle \nabla \phi(z), x-z \rangle.
	\]
	For example in the case of the simplex, $D_\phi$ is the \emph{Kullback-Leiber divergence}.
	Since \(\phi\) is finite in \(\cal X\) and differentiable over \({\rm ri}({\cal X})\), \(D_\phi(x,z)\) is defined at least for \(x\in \cal X\) and \(z\in {\rm ri}({\cal X})\). It is well known (see \cite[Remark 9.6]{B17}) that \eqref{eq:mirrordescent} is equivalent to
	\[
	x_{k+1}=\argmin_{x\in\mathcal{X}}  \left \{f(x_{k})+\inner{\nabla f(x_{k}),x-x_{k}}+\frac{1}{h}D_{\phi}(x,x_{k}) \right \}    .
	\]
	which is the direct non Euclidean counterpart of the gradient descent formula \eqref{eq:maxmin}.
	
	\section{Non-Euclidean accelerated ODEs and numerical schemes}\label{sec:non-Euclidean}
	
	Mirror descent may only achieve a \(\mathcal{O}(1/k)\) rate of convergence \citep{B17}, something to be expected from the fact that in the Euclidean setting it reduces to gradient descent. We now consider ODEs and algorithms that may provide rates \(\mathcal{O}(1/t^2)\) or \(\mathcal{O}(1/k^2)\)  in non Euclidean scenarios.
	
	\subsection{A Primal/Dual ODE and a discretization} \label{subsec:pr_du_ODE}
	
	In order to construct optimization algorithms,  \cite{KBB15} considered the system
	\begin{subequations}\label{eq:Kr_ode}
		\begin{eqnarray}\label{eq:odez}
			\dot{\zeta}(t) &=& -\frac{t}{r} \nabla f(x(t)),\\
			\label{eq:odex}
			\dot{x}(t) &=& \frac{r}{t} \big(\chi(\zeta(t))-x(t)\big),
		\end{eqnarray}
	\end{subequations}
	where \(r>0\) is a parameter. The variable \(x\) takes values in \(\mathbb{E}\) and \(\zeta\) takes values in \(\mathbb{E}^\star\); the paper \citep{KBB15} proves that if the initial data \((\zeta_0,x_0) \in \mathbb{E}^\star\times {\cal X }\) satisfy\footnote{The requirement \(\chi(\zeta_0)=x_0\) is imposed in view of the singularity of \eqref{eq:odez} at \(t=0\).} \(\chi(\zeta_0)=x_0\), then the system has a unique continuously differentiable solution \(\big(\zeta(t),x(t)\big)\) for \(0\leq t<\infty\). Note that \eqref{eq:Kr_ode} is a natural generalization of \eqref{eq:Polyak_first} as in fact coincides with it in the Euclidean case $\chi(\zeta)=\zeta$.

	If, \(r\geq 2\), \(x^\star\) is a minimizer,  \(\chi(\zeta^\star)= x^\star\), and Assumption \ref{as:cond_mir1} is satisfied, then
	\begin{equation}\label{eq:V}
		V\big(x,\zeta,t\big) = \frac{t^2}{r^2} \big(f(x)-f(x^\star)\big)+ D_{\psi^\star}\big(\zeta,\zeta^\star\big)\geq 0
	\end{equation}
	is a Lyapunov function for the system, i.e.\ \((d/dt)V(x(t),\zeta(t),t)\leq 0\) along solutions of  \eqref{eq:Kr_ode}. This immediately implies the following decay estimate of \(f(x(t))\) towards the optimal value \(f(x^\star)\)
	\begin{equation*}
		f(x(t))-f(x^\star) \leq \frac{r^2}{t^2} V\big(x(0),\zeta(0),0\big)= \frac{r^2}{t^2}
		D_{\psi^\star}\big(\zeta(0),\zeta^\star\big), \qquad t>0.
	\end{equation*}Note for future reference that \(V\) includes the dual variable \(\zeta\) through \(D_{\psi^\star}\) as in \eqref{eq:dstar}.

	As in the situation   we discussed in the Euclidean case,  discretizations \(\big(\zeta_k,x_k\big)\mapsto \big(\zeta_{k+1},x_{k+1}\big)\) of \eqref{eq:Kr_ode}, where \(\big(\zeta_k,x_k\big)\) approximate \(\big(\zeta(k\delta),x(k\delta)\big)\) (\(\delta>0\) is the time-step), may offer the potential of providing optimization algorithms for which \(f(x_k)-f(x^\star)\) decays like \(1/k^2\), i.e.\ algorithms that show acceleration.
	An  algorithm with acceleration was suggested in \cite{KBB15}. It uses a learning rate \(h>0\),  parameters \(r>0\) and \(\gamma>0\) and  a \emph{regularization} function \(R\) such that for \(x\), \(y\in \cal X\),
	\begin{equation}\label{eq:hypothesesonR}
	\frac{\ell_R}{2} \| x-y\|^2 \leq R(x,y)\leq \frac{L_R}{2}\| x-y\|^2.
	\end{equation}
	We will refer to it as accelerated mirror descent with regularization (AMDR) and it is given in Algorithm~\ref{alg:amd_kr}.
	\begin{algorithm}
		\caption{Accelerated Mirror Descent with Regularization (AMDR) \citep{KBB15}}\label{alg:amd_kr}
		\begin{algorithmic}
			\Require $ N\in \mathbb{N}, h>0, r\geq 0, \gamma>0$ and regularizer $R$
			
			\State 0.  Initialize: $x_0\in \mathcal{X}, \zeta_0\in \mathbb{E}^\star \big(\chi(\zeta_0) = x_0\big)$
			\For{$k=0,\ldots, N-1$}
			\State 1. $y_k \gets x_k +\frac{r}{r+k} \big( \chi(\zeta_k)-x_k)$
			\State 2. $\zeta_{k+1} \gets  \zeta_k -\frac{kh}{r}\nabla f(y_k)$  
			\State 3. $x_{k+1} \gets {\rm arg min}_{x\in \cal X}\big(\gamma h \langle \nabla f(y_k),x \rangle+R(x,y_k ) \big)$
			\EndFor\\
			\Return $x_N$
		\end{algorithmic}
	\end{algorithm}
	If \(\delta = \sqrt{h}\), then the algorithm may be seen as a numerical method to integrate the system \eqref{eq:Kr_ode}, with \(\zeta_k\) and \(x_k\) approximations to \(\zeta(k\delta)\) and \(x(k\delta)\) respectively. This is easily proved after taking into consideration that \(x_{k+1}\) and \(y_k\) differ by an \(\mathcal{O}(\delta^2)\) amount \citep{KBB15}.
	
	The discretization in AMDR  was constructed so as to \emph{inherit} the Lyapunov function \eqref{eq:V}. This is a nontrivial task, because, typically, numerical integrators, even if very accurate, fail to reproduce the large \(t\) properties of the system being integrated; see the discussion in \cite{SSKZ21}. For this algorithm, it is proved in \cite[Lemma 2]{KBB15} that, if \(\gamma \geq L_RL_\chi\) and \(h\leq \ell_R/(2L_f \gamma)\), then
	\[
	V(x_{k+1},\zeta_{k+1}, (k+1)\delta) - V(x_k,\zeta_k, k\delta)\leq \frac{(2k+1-kr)}{r^2} \big(f(x_{k+1})-f(x^\star)\big).
	\]
	For \(r\geq 3\), \(k\geq 1\), the right hand-side is \(\leq 0\) and thus the bound establishes an \(\mathcal{O}(1/k^2)\) decay of \(f(x_k)-f(x^\star)\) (acceleration).

	\subsection{An alternative primal/dual discretization}
	\label{subsec:alg_new}
	The need for the regularisation function $R$ in AMDR could be problematic. For example, in the case of the simplex, one possible choice of $R$ is  an $\epsilon$-smooth entropy function \citep{KBB15,KKB15}. In that case, there is an efficient algorithm to implement Step 3 of AMDR, but unfortunately the value of $\gamma$ to be used depends on $\epsilon$ and  the learning rate $h$ can become prohibitively small {\color{black} see also discussion in Remark \ref{rem:lr}}. Furthermore, beyond the simplex setting, it might not be obvious how to set $R$ and implement Step 3 {\color{black} (see Section~\ref{sec:numerics})}.  Motivated by this, we propose Algorithm \ref{alg:amd_us} (AMD) that is a natural generalization of Nesterov's method in the non-Euclidean setting and makes no use of a regularization step. A learning rate \(h>0\) and a sequence \(\{\gamma_k\}_{k=0}^\infty\) with \(\gamma_0=1\), \(\gamma_k \geq 1\),  \(k= 1, 2,\dots\), are required.
	\begin{algorithm}
		\caption{Accelerated Mirror Descent (AMD)}\label{alg:amd_us}
		\begin{algorithmic}
			\Require $ N\in \mathbb{N}, \{\gamma_k\}_{k=0}^{N-1}, h>0$
			\State 0. Initialize $x_0\in \mathcal{X}$, $\zeta_0\in \mathbb{E}^\star$,
			\For{$k=0,\ldots, N-1$}
			\State 1. $y_k \gets x_k +\frac{1}{\gamma_k} \big( \chi(\zeta_k)-x_k\big)$
			\State 2. $\zeta_{k+1} \gets \zeta_k -\gamma_k h\nabla f(y_k)$  
			\State 3. $x_{k+1} \gets y_k +\frac{1}{\gamma_k}\big(\chi(\zeta_{k+1})-\chi(\zeta_k)\big)$
			\EndFor\\
			\Return $x_N$
		\end{algorithmic}
	\end{algorithm}

	Note that Step 3 is equivalent to
	\begin{equation}\label{eq:algxbis}
		x_{k+1} = x_k +\frac{1}{\gamma_k}\big(\chi(\zeta_{k+1})-x_k\big).
	\end{equation}
	In Step 1, $y_k$ is a convex combination of \(x_k\) and \(\chi(\zeta_k)\) and in \eqref{eq:algxbis} $x_{k+1}$ is a convex combination of \(x_k\) and \(\chi(\zeta_{k+1})\). By induction, all the \(y_k\) are in \(\cal X\) (so that \(\nabla f(y_k)\) makes sense) and all the \(x_k\) are also in \(\cal X\).
	
	In the Euclidean case this algorithm reduces to
	\eqref{eq:nest_euclidean}, i.e.\ to the Nesterov algorithm implemented as a one-step recursion with the help of the variable \(z\). The standard implementation \eqref{eq:euclideo1}--\eqref{eq:euclideo2} of Nesterov's algorithm   cannot be directly applied to the non Euclidean scenario since \eqref{eq:euclideo1} mixes primal and dual variables. In  AMD such a mixing is avoided; in Step 2, the gradients are accumulated in a dual variable (as in mirror descent \eqref{eq:descentdual}) and the primal mirror images of the dual variable are used to perform the convex combinations in Steps 1 and 3.

\begin{remark}
		\color{black}{	In AMDR there is freedom to select an appropriate regularizer $R$.  The recommendation in \cite{KBB15} is to use a Bregman divergence $R=D_g$ for a suitable function $g$. In that case Step 3 of AMDR is a mirror descent step and can be written in the same form as \eqref{eq:mirrordescent}, i.e.
			$
			x_{k+1} = \chi_g(\nabla g(y_k)-\gamma h \nabla f(y_k))
			$
			where $\chi_g = \nabla \psi_g^\star$ with $\psi_g= g+\delta_\mathcal{X}$.
			We can rewrite this expression in terms of $\zeta_{k+1}$ and $\zeta_k$ as
			\begin{equation*}
				x_{k+1} = \chi_g\left(\nabla g(y_k)+\frac{\gamma r}{k}(\zeta_{k+1}-\zeta_k) \right).
			\end{equation*}
			This rewriting may make it easier to compare Step 3 in AMDR and AMD: while AMD performs a linear combination on the primal space, for AMDR the linear combination is performed on the dual space. This means that in general, even if we allow $g=\phi$ where $\phi$ is taken from  Assumption~\ref{ass:two}, the two algorithms will not agree except for the unconstrained Euclidean case.
			
			Note that in most applications (for instance, the simplex and the hypercube) the choice $g=\phi$ is not covered by the analysis of \cite{KBB15}. Indeed, the restriction that \eqref{eq:hypothesesonR} holds, i.e.\ that $g$ must be strongly convex and smooth with respect to the norm $\lVert \cdot \rVert$, means that choices such as negative entropy or negative bit entropy are not permissible. On the other hand, AMDR may benefit from the advantage of  choosing $\chi$ based purely on the function $f$, so as to work with a normed space $\bbE$ for which $L_f$ is small, while simultaneously choosing $R$ to deal with the set $\mathcal{X}$.}
	\end{remark}

	After defining \(\delta = \sqrt{h}\) and \(\tilde t_k = r\delta\gamma_k\), Steps 1 and 2 in Algorithm \ref{alg:amd_us} and \eqref{eq:algxbis} imply
	\begin{eqnarray*}
		\frac{1}{\delta} (y_{k}-x_k) & = & \frac{r}{\tilde t_k}(\chi(\zeta_k)-x_k),\\
		\frac{1}{\delta}(\zeta_{k+1}-\zeta_k) &=& - \frac{\tilde t_k}{r} \nabla f(y_k),\\
		\frac{1}{\delta} (x_{k+1}-x_k) & = & \frac{r}{\tilde t_k}(\chi(z_{k+1})-x_k).
	\end{eqnarray*}
	If we now assume that
	\begin{equation}\label{eq:asympgamma}
		\gamma_k = \frac{k}{r}+o(k), \qquad k\rightarrow \infty,
	\end{equation}
	then \(\tilde t_k \rightarrow k\delta\)  as \(k\rightarrow \infty\) {\color{black} and} \(\delta \rightarrow 0\) with \(k\delta\) constant.
	We therefore have the following result:
	\begin{theorem} Suppose that Assumption~\ref{as:cond_mir1} and \eqref{eq:asympgamma}  hold, then AMD provides a consistent one-step numerical integrator  \((\zeta_k,x_k)\mapsto (\zeta_{k+1},x_{k+1})\) for the system of ODEs \eqref{eq:Kr_ode}.
	\end{theorem}

	The following result implies that, for suitable choices of the learning rate \(h\) and the constants \(\gamma_k\),  AMD is indeed an accelerated optimization method.

	\begin{theorem}\label{th:decay} Suppose that Assumption~\ref{as:cond_mir1}  holds and that the coefficients \(\gamma_k\) satisfy
		\begin{equation}\label{eq:upperbound}
			\gamma_k^2-\gamma_{k-1}^2-\gamma_k\leq 0, \qquad k=1,2,\dots
		\end{equation}
		Assume that \(f\) is \(L_f\)-smooth, i.e.\ for each \(x,y\) in the domain of \(f\):
		\begin{equation}\label{eq:smooth}
			\|\nabla f(x)-\nabla f(y)\|_\star \leq L_f\|x-y\|
		\end{equation}
		and that
		\begin{equation}\label{eq:h}
			h\leq \frac{1}{L_f L_\chi}.
		\end{equation}
		Let \(x^\star\) be a minimizer of \eqref{eq:min} and \(\zeta^\star\in\mathbb{E}^\star\) satisfy \(\chi(\zeta^\star)\), and define for \(k=0,1,\dots\)
		\begin{equation}
			\label{eq:lyap}V_k(x_k,\zeta_k) = (\gamma_k^2-\gamma_k)h \big(f(x_k)-f(x^\star)\big)+ D_{\psi^\star}(\zeta_k,\zeta^\star).
		\end{equation}
		Then, along trajectories generated by AMD
		\[
		V_{k+1}(x_{k+1},\zeta_{k+1})\leq V_k(x_k,\zeta_k), \qquad k = 0,1,\dots
		\]
		and therefore
		\[
		(\gamma_k^2-\gamma_k)h\big(f(x_k)-f(x^\star)\big) \leq D_{\psi^\star}(\zeta_0,\zeta^\star), \qquad k=0,1,\dots
		\]
	\end{theorem}

\begin{proof}
	We first deal with the part of the Lyapunov function that involves the variable \(x\).
	By the convexity and smoothness of \(f\),
	\[
	f(x_{k+1}) \leq f(y_k) +\langle\nabla f(y_k), x_{k+1}-y\rangle + \frac{L_f}{2} \|x_{k+1}-y_k\|^2,
	\]
	and, from the definition of \(x_{k+1}\) in Step 3 of Algorithm~\ref{alg:amd_us}, after shortening slightly the notation,
	\[
	f(x_{k+1}) \leq f(y_k) +\langle\nabla_k , x_{k+1}-y_k\rangle + \frac{L_f}{2\gamma_k^2} \|\chi_{k+1}-\chi_k\|^2.
	\]
	Here $\nabla_k=\nabla f(y_k)$ and $\chi_k=\chi(\zeta_k)$.
	We now use \eqref{eq:algxbis} to get
	\begin{eqnarray*}
		f(x_{k+1}) &\leq& f(y_k) +\langle\nabla_k , \left(1-\frac{1}{\gamma_k}\right)x_k+\frac{1}{\gamma_k}\chi_{k+1}-y_k\rangle + \frac{L_f}{2\gamma_k^2} \|\chi_{k+1}-\chi_k\|^2\\
		&=&  \left(1-\frac{1}{\gamma_k}\right)\Big(f(y_k)+\langle \nabla_k, x_k-y_k\rangle\Big) \\
		&&\qquad+ \frac{1}{\gamma_k} \Big(f(y_k)+\langle \nabla_k, \chi_{k+1}-y_k\rangle\Big)
		+ \frac{L_f}{2\gamma_k^2} \|\chi_{k+1}-\chi_k\|^2,\\
		&=& \left(1-\frac{1}{\gamma_k}\right)\Big(f(y_k)+\langle \nabla_k, x_k-y_k\rangle\Big) + \frac{1}{\gamma_k} \Big(f(y_k)+\langle \nabla_k, x^\star-y_k\rangle\Big)
		\\
		&&\qquad + \frac{1}{\gamma_k} \langle \nabla_k, \chi_{k+1}-x^\star\rangle+\frac{L_f}{2\gamma_k^2} \|\chi_{k+1}-\chi_k\|^2.
	\end{eqnarray*}
	Invoking again the convexity of \(f\)
	\begin{eqnarray*}
		f(x_{k+1})&\leq& \left(1-\frac{1}{\gamma_k}\right)f(x_k)+ \frac{1}{\gamma_k} f(x^\star)+ \frac{1}{\gamma_k} \langle \nabla_k, \chi_{k+1}-x^\star\rangle\\
		&&\qquad +\frac{L_f}{2\gamma_k^2} \|\chi_{k+1}-\chi_k\|^2,
	\end{eqnarray*}
	and therefore
	\begin{eqnarray*}
		f(x_{k+1})-f(x^\star) &\leq& \left(1-\frac{1}{\gamma_k}\right) \big(f(x_k)-f(x^\star)\big)+ \frac{1}{\gamma_k} \langle \nabla_k, \chi_{k+1}-x^\star\rangle\\
		&& \qquad+\frac{L_f}{2\gamma_k^2} \|\chi_{k+1}-\chi_k\|^2.
	\end{eqnarray*}
	We now multiply across by \(\gamma_k^2h\) and take into account the formula in Step 2 of the algorithm,
	\begin{eqnarray*}
		\gamma_k^2h\big(f(x_{k+1})-f(x^\star)\big) &\leq& (\gamma_k^2-\gamma_k)h \big(f(x_k)-f(x^\star)\big)\\
		&& \qquad-\langle \zeta_{k+1}-\zeta_k, \chi_{k+1}-x^\star\rangle+\frac{L_fh}{2} \|\chi_{k+1}-\chi_k\|^2.
	\end{eqnarray*}
	The bounds \eqref{eq:upperbound} and \eqref{eq:h} then yield
	\begin{eqnarray}
		\label{eq:potentialdecay}
		&&(\gamma_{k+1}^2-\gamma_{k+1})h \big(f(x_{k+1})-f(x^\star)\big) \:\leq \:(\gamma_k^2-\gamma_k)h \big(f(x_k)-f(x^\star)\big)\\&&\nonumber
		\qquad\qquad\qquad\qquad\qquad\qquad-\langle \zeta_{k+1}-\zeta_k, \chi_{k+1}-x^\star\rangle +\frac{1}{2L_\chi} \|\chi_{k+1}-\chi_k\|^2.
	\end{eqnarray}
	
	We next address the part of the Lyapunov function involving the variable \(\zeta\).
	By using the three-points lemma  for \(D_{\psi^\star}\) as in \cite[Lemma 5]{KBB15} and recalling that \(\chi(\zeta^\star) =x^\star\),
	\[
	D_{\psi^\star}(\zeta_{k+1},\zeta^\star)=D_{\psi^\star}(\zeta_k,\zeta^\star)-D_{\psi^\star}(\zeta_k,\zeta_{k+1})+\langle \zeta_{k+1}-\zeta_k,\chi_{k+1}-x^\star \rangle,
	\]
	and the smoothness bound in \cite[Lemma 5]{KBB15} yields
	\[
	D_{\psi^\star}(\zeta_{k+1},\zeta^\star)\leq D_{\psi^\star}(\zeta_k,\zeta^\star) -\frac{1}{2L_\chi}\|\chi_{k+1}-\chi_k\|^2 +\langle \zeta_{k+1}-\zeta_k,\chi_{k+1}-x^\star \rangle.
	\]
	It is now sufficient to add the last bound to \eqref{eq:potentialdecay}.
\end{proof}
	
	{\color{black} \begin{remark} \label{rem:lr} The learning rate $h$ allowed in the case of AMD so as to ensure that the corresponding Lyapunov function is indeed decaying is larger than the corresponding one in the case of AMDR. In particular, for the case of the simplex on $\mathbb{R}^{n}$, where the regularization used in AMDR is the $\epsilon$-smooth entropy function ones has that for AMD the upper bound is $h_{up}=\frac{1}{L_{f}}$, since $L_{\chi}=1$,  while in the case of AMDR one has $h_{up}=\frac{\epsilon}{2(1+n\epsilon)L_{f}}$. One thus would expect that if these bounds are sharp that AMDR would be behaving  slower than AMD (see also the numerical experiments in Section \ref{subsec:quad_obj}).  \end{remark}}
	
	Under the consistency condition \eqref{eq:asympgamma}, the right hand-side of \eqref{eq:lyap} converges, in the limit \(\delta\rightarrow 0\), \(k\delta \rightarrow t\), to the Lyapunov function \eqref{eq:V} of the differential equations. The choice
	\begin{equation}\label{eq2:choicer}
		\gamma_k = \frac{k+r}{r}, \qquad k = 0, 1,\dots
	\end{equation}
	fulfills the consistency requirement \eqref{eq:asympgamma} and, if \(r\geq 2\), also the condition \eqref{eq:upperbound}.  Recall that the same condition on \(r\) is required for \eqref{eq:V}  to be a Lyapunov function for the differential equations unlike the case of AMDR. When the coefficients \(\gamma_k\) are chosen as in  \eqref{eq2:choicer}, the theorem shows a decay \(f(x_k)-f(x^\star)\) like \(1/((\gamma_k^2-\gamma_h)h)\sim r^2/(k^2h)\) as \(k\rightarrow \infty\).
	
	The best decay of \(f(x_k)-f(x^\star)\) that may be proved with the theorem occurs when the \(\gamma_k\) are chosen as large as possible subject to \eqref{eq:upperbound}, i.e.\ when the inequality in \eqref{eq:upperbound} becomes an equality. In this case, we have the well-known recurrence \eqref{eq:fancygamma}.
	These coefficients are slightly larger than \eqref{eq2:choicer} with \(r=2\) and, accordingly, guarantee a slightly better convergence.  One may prove that with this recurrence, as \(k\rightarrow \infty\),  \(\gamma_k = k/2+(1/4)\log k+o(\log k)\), to be compared with the estimate
	\(\gamma_k = k/2+\mathcal{O}(1)\) valid for \eqref{eq2:choicer} with \(r=2\).

	\subsubsection{Connection with Additive Runge Kutta methods}
	{\color{black}
	As discussed in the previous section, AMD  is a discretization of the  ODE \eqref{eq:Kr_ode}. However, this discretization does not correspond to any of the more standard classes of ODE solvers, such as linear multistep or Runge-Kutta (RK) methods. {\color{black} Note in particular that $\nabla f$ is not evaluated at the approximations $\zeta_{k}$ produced by the algorithm as it would typically be the case with classical methods.} As we will show similarly to the case of Nesterov algorithms for strongly convex objective functions in Euclidean space  \cite{DSZ24}, AMD is an example of the class of Additive Runge-Kutta (ARK) algorithms, a generalization of the RK integrators firstly introduced in  \citep{Cooper80,Cooper83}.
	
	ARK algorithms integrate systems of differential equations $(d/dt)\xi=g(\xi)$ in cases where $g(\xi)$ can be decomposed as a sum $g(z)=\sum_{i=1}^{N}g^{[i]}(\xi)$. In the plain RK case, the numerical solution is advance over one time step $\xi_{k} \mapsto \xi_{k+1}$ by evaluating $g(\xi)$ at a sequence of so-called stage vectors $\Xi_{k,1} \dots, \Xi_{k,s}$  and then setting $z_{k+1}=z_{k}+\sum_{i=1}^{s}b_{i}g(\Xi_{k,i})$, where $b_{i}$ are suitable weights. In the case of explicit algorithms the stages are computed successively, $i=1,\cdots,s$ as $\Xi_{k,i}=\xi_{k}+h\sum_{j=1}^{i-1} a_{i,j}g(\Xi_{k,j})$, with suitable coefficents $a_{i,j}$. ARK algorithms are entirely similar, but evaluate the individual pieces $g^{[i]}(\xi)$ instead  of $g(\xi)$.
	}

	 We  now set \( \xi = (\zeta,x)\) and write
	\eqref{eq:Kr_ode} by additively decomposing the right hand-side as
	\[
	\dot{\xi} = g^{[1]}(\xi,t)+g^{[2]}(\xi,t)+g^{[3]}(\xi,t),
	\]
	with
	\[
	g^{[1]}(\xi,t) = \left[ \begin{matrix}0\\-\frac{r}{t} x\end{matrix}\right], \quad
	g^{[2]}(\xi,t) = \left[ \begin{matrix}0\\\frac{r}{t} \chi(\zeta)\end{matrix}\right], \quad
	g^{[3]}(\xi,t) = \left[ \begin{matrix}-\frac{t}{r} \nabla f(x)\\0\end{matrix}\right].
	\]
	Then the step \((\zeta_k,x_k)\mapsto (\zeta_{k+1},x_{k+1})\) may be written in a Runge-Kutta fashion as
	\[
	\xi_{k+1} = \xi_{k} +\delta g^{[1]}(\Xi_{k,1},\tilde t_k)+\delta g^{[2]}(\Xi_{k,3},\tilde t_k)+\delta g^{[3]}(\Xi_{k,3},\tilde t_k)
	\]
	with the stage vectors defined by
	\begin{eqnarray*}
		\Xi_{k,1} &=& \xi_k,\\
		\Xi_{k,2} &=& \xi_k+\delta  g^{[1]}(\Xi_{k,1},\tilde t_k)+\delta g^{[2]}(\Xi_{k,1},\tilde t_k),\\
		\Xi_{k,3} &=& \xi_k+\delta  g^{[1]}(\Xi_{k,1},\tilde t_k)+\delta g^{[2]}(\Xi_{k,1},\tilde t_k)+\delta g^{[3]}(\Xi_{k,2},\tilde t_k).
	\end{eqnarray*}
	Note that  \(\Xi_{k,1} = (\zeta_k,x_k)\), \(\Xi_{k,2} = (\zeta_k,y_k)\), \(\Xi_{k,3} = (\zeta_{k+1},y_k)\). Thus the successive computations of  \(\Xi_{k,2}\), \(\Xi_{k,3}\), and \(\xi_{k+1}\) in the Additive Runge-Kutta scheme represent the computations of \(y_k\), \( \zeta_{k+1}\), \(x_{k+1}\).
	
	\subsection{Convergence  when \(x^\star\) is not in the image of the mirror map}\label{subsec:alternative}

	The Lyapunov functions \eqref{eq:V} \eqref{eq:lyap}, used  for establishing convergence for  the ODE \eqref{eq:Kr_ode} and its discretizations AMDR and AMD, contain a term \(D_{\psi^\star}(\zeta,\zeta^\star)\), where \(\chi(\zeta^\star)=x^\star\). Therefore they cannot be used to establish convergence when the minimizer \(x^\star\) is not in the image of the mirror map. In the case of the simplex, this implies that one cannot treat minimizers \(x^\star\) having one or more zero components. Similarly, for the hypercube, minimizers having some of their components equal to \(0\) or \(1\) cannot be dealt with. In this subsection we remove this limitation by using Lyapunov functions that, as distinct from those considered above or in \cite{KBB15}, are formulated purely in terms of \emph{primal} variables. Accordingly,  {\color{black} similarly to \cite{Merti}}, we will operate with \emph{Bregman divergences} defined in \(\bbE\) rather than in \(\bbE^\star\) and this will require that Assumption~\ref{ass:two} holds.

	\subsubsection{The differential system}
	
	Using the primal Bregman divergence \(D_\phi\), for the system of differential equations \eqref{eq:Kr_ode}, in lieu of the Lyapunov function \eqref{eq:V}, we may alternatively consider
	\begin{equation}\label{eq:Vhat}
		\widehat V\big(x,\zeta,t\big) = \frac{t^2}{r^2} \big(f(x)-f(x^\star)\big)+ D_\phi\big(x^\star,\chi(\zeta)),
	\end{equation}
	where we note that \(D_\phi\big(x^\star,\chi(\zeta))\)  is well defined because \(\chi\) takes values in \({\rm ri}({\cal X})\). Now
	the existence of \(\zeta^\star\) with \(\chi(\zeta^\star) = x^\star\) is not required. If such a \(\zeta^\star\) exists, then the numerical values of \eqref{eq:V} and \eqref{eq:Vhat} coincide, according to well-known properties of the Bregman divergence. The following theorem shows that, for \(r\geq 2\), \(\widehat V\) is indeed a Lyapunov function and therefore \(f(x(t))-f(x^\star)\) decays like \(1/t^2\).
	\begin{theorem}\label{th:continuoustime} Suppose that Assumptions~\ref{as:cond_mir1} and ~\ref{ass:two} hold. If \(r\geq 2\), then along solutions of \eqref{eq:Kr_ode}, \((d/dt)\widehat{V} \leq 0\).
	\end{theorem}

	\begin{proof}
	By differentiating we have,
	\begin{eqnarray*}
		\frac{d}{dt} \widehat V &=& \frac{2t}{r^2} \big(f(x(t))-f(x^\star)\big)-\frac{t^2}{r^2}\langle\nabla f(x(t)), \dot{ x}(t)\rangle\\&&
		\qquad+
		\langle\frac{d}{dt}\nabla \phi\big(\chi(\zeta(t))\big),x^\star-\chi(\zeta(t)) \rangle,
	\end{eqnarray*}
	and Lemma~\ref{lemma}, Part 3, implies
	\[
	\langle\frac{d}{dt}\nabla \phi\big(\chi(\zeta(t))\big),x^\star-\chi(\zeta(t)) \rangle = \langle\dot{\zeta}(t),x^\star-\chi(\zeta(t)) \rangle.
	\]
	These two equalities and the differential equations \eqref{eq:odez}--\eqref{eq:odex} may be combined to yield:
	\begin{eqnarray*}
		\frac{d}{dt}\widehat V &=&\left(\frac{2}{r}-1\right)\frac{t}{r}\big(f(x(t))-f(x^\star)\big)\\
		&&\qquad -\frac{t}{r} \ \Big(f(x^\star) -f(x(t))-\langle\nabla f(x(t)) ,x^\star-x(t)\rangle\Big).
	\end{eqnarray*}
	Both terms in the right hand-side are \(\leq 0\), the first because \(r\geq 2\) and the second because \(f\) is convex.
\end{proof}
	
	\subsubsection{Algorithm~\ref{alg:amd_kr} (AMDR)}
	
	For  AMDR, Lemma 2 in \cite{KBB15} may be replaced by the following new result which implies that for \(r\geq 3\) we shall have acceleration even if \(x^\star\) is not in the image of \(\chi\).
	\begin{theorem}\label{th:decaykrichene} Suppose that Assumptions~\ref{as:cond_mir1} and ~\ref{ass:two} hold.
		If \(f\) is \(L_f\)-smooth,\(\gamma \geq L_RL_\chi\) and \(h\leq \ell_R/(2L_f \gamma)\), then for AMDR
		\[
		\widehat V(x_{k+1},\zeta_{k+1}, (k+1)\delta) -\widehat V(x_k,\zeta_k, k\delta)\leq \frac{(2k+1-kr)h}{r^2} \big(f(x_{k+1})-f(x^\star)\big).
		\]
	\end{theorem}

\begin{proof}The part of the Lyapunov function involving \(x\) is dealt with as in the proof of Lemma 2 in \cite{KBB15} and the difference in Bregman divergences is treated exactly as in the proof of Theorem~\ref{th:decaybis} to be given below.
\end{proof}
	
	\subsubsection{Algorithm~\ref{alg:amd_us} (AMD)}
	In the context of the AMD instead of using the discrete Lyapunov function \eqref{eq:lyap}, we will alternatively consider for \(k= 0,1,\dots\)
	\begin{equation}
		\label{eq:lyapbis}\widehat {V}_k(x_k,\zeta_k) = (\gamma_k^2-\gamma_k)h \big(f(x_k)-f(x^\star)\big)+ D_\phi\big(x^\star,\chi(\zeta_k)\big).
	\end{equation}
	By using
	\(\widehat V\), Theorem~\ref{th:decay}  may be strengthened as follows {\color{black}to show that the decay of \(f(x_k)-f(x^\star)\) takes place whether \(x^\star\) is in the image of \(\chi\) or otherwise.}
:
	
	\begin{theorem}\label{th:decaybis}
		Suppose that Assumptions~\ref{as:cond_mir1} and  \ref{ass:two} hold and that the coefficients \(\gamma_k\) satisfy \eqref{eq:upperbound}.
		Assume that \(f\) is \(L_f\)-smooth
		and that
		\[
		h\leq \frac{1}{L_f L_\chi}.
		\]

		Let \(x^\star\) be a minimizer of \eqref{eq:min}. Then, for $(x_{k+1},\zeta_{k+1})$ given by AMD
		\[
		\widehat V_{k+1}(x_{k+1},\zeta_{k+1})\leq \widehat V_k(x_k,\zeta_k), \qquad k = 0,1,\dots
		\]
		and therefore
		\[
		(\gamma_k^2-\gamma_k)h\big(f(x_k)-f(x^\star)\big) \leq D_\phi(x^\star,\chi(\zeta_0)), \qquad k=0,1,\dots
		\]
	\end{theorem}
	
\begin{proof}We reproduce the proof of Theorem~\ref{th:decay} until we reach  \eqref{eq:potentialdecay}. For the part involving Bregman divergences, the three-points lemma for \(D_\phi\) gives
	\begin{eqnarray*}
		D_\phi(x^\star,\chi_{k+1})&=&D_\phi(x^\star,\chi_k)-D_\phi(\chi_{k+1},\chi_k)\\&&\qquad\qquad+\langle \nabla\phi(\chi(\zeta_{k+1}))-\nabla\phi(\chi(\zeta_k)), \chi_{k+1}-x^\star\rangle
	\end{eqnarray*}
	and, as a consequence of Lemma~\ref{lemma}, Part 3, the differences \(\nabla\phi(\chi(\zeta_{k+1}))-\zeta_{k+1}\) and \(\nabla\phi(\chi(\zeta_k))-\zeta_k\) are in \(\cal N\)
	and we may alternatively write
	\[
	D_\phi(x^\star,\chi_{k+1})=D_\phi(x^\star,\chi_k)-D_\phi(\chi_{k+1},\chi_k)+\langle \zeta_{k+1}-\zeta_k, \chi_{k+1}-x^\star\rangle.
	\]
	From \cite[Lemma 9.4 (a)]{B17}
	\[
	D_\phi(x^\star,\chi_{k+1})\leq D_\phi(x^\star,\chi_k)-\frac{1}{2L_\chi}\|\chi_{k+1}-\chi_k\|^2+\langle \zeta_{k+1}-\zeta_k, \chi_{k+1}-x^\star\rangle,
	\]
	because \(\psi=\phi+\delta_{\cal X}\) is \((1/L_\chi)\)-strongly convex  \cite[Theorem 5.26(a)]{B17}.
	The proof concludes by adding the last bound to \eqref{eq:potentialdecay}.
\end{proof}

\section{Numerical experiments}\label{sec:numerics}
	
	We now illustrate the performance of  AMDR  and AMD.  {\color{black} A full comparison would be lengthy ---as it would require  investigating different choices of $\cal X$, $f$ regularizers, etc.--- and is  not within the scope of this paper.}
The standard mirror descent algorithm \eqref{eq:mirrordescent} will be used as a benchmark. {\color{black} For AMD we use \(\gamma_k\) given by \eqref{eq:fancygamma}.}

\subsection{Examples for the simplex}
	In this Subsection we consider, as in \cite{KBB15}, problems where \(\cal X\) is the simplex.
	Recall that in this case it is possible to run AMDR with an efficient regularizer, something that, {\color{black} as we shall discuss later,} may or may not be the situation for other instances of \(\cal X\).
	
	In the experiments that follow, we set \(r=3\) for AMDR and, as in \cite{KBB15},  use \(\gamma = 1\) and perform Step 3 by means of the efficient procedure in \cite[Algorithm 4]{KBB15} with \(\epsilon = 0.3\). It turns out that with this setting, the computational costs per step of AMDR  and AMD are virtually identical and also coincide with those of mirror descent.
	
	\subsubsection{Non strongly convex objective function}
	\begin{figure}
		\centering
		{\includegraphics[width=0.8\textwidth]{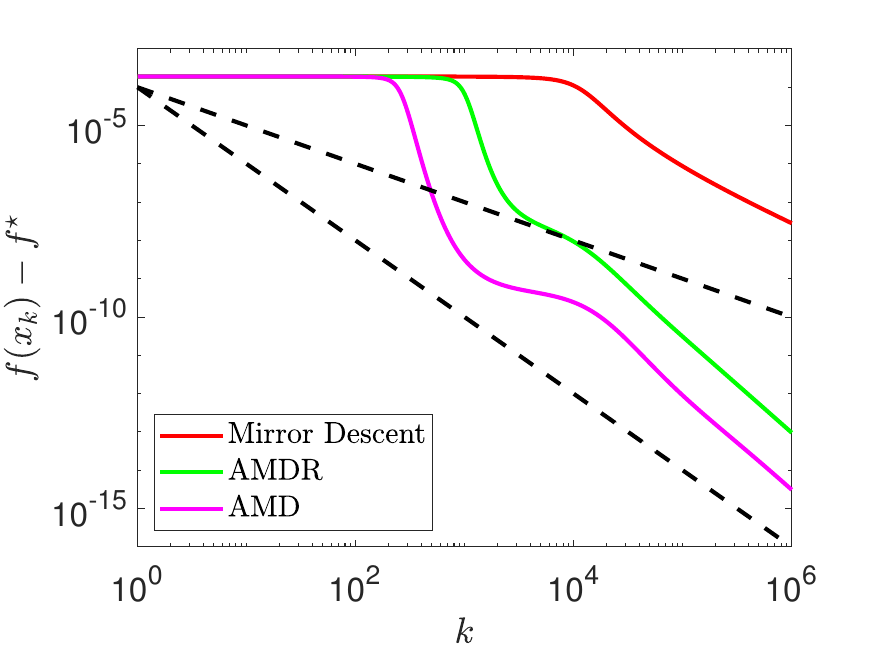}}
		\caption{Non strongly convex objective function, $f(x_k)-f(x^\star)$ vs. $k$. The dotted lines have  slopes corresponding to  decays $1/k$ and $1/k^2$.}
		\label{fig:firstexperiment}
	\end{figure}
	
	In order to check that  AMDR and AMD provide acceleration, we first consider an extremely  simple toy example with \(d=2\), \(f(x) = (1/p)\big((x_1-1/2)^{p}+(x_2-1/2)^{p}\big)\), \(p=10\). The initial condition is chosen as \([0.999, 0.001]^T\) and the three algorithms were run with different choices of the learning rate. Results for the representative value \(h=1\)  may be seen in Figure~\ref{fig:firstexperiment}. While for mirror descent, the decay is slightly better than \(1/k\), for AMDR and AMD the decay is slightly better than \(1/k^2\). Increasing the value of the parameter \(p\) results in rates that become closer to \(1/k\) for mirror descent and to \(1/k^2\) for the other two algorithms.
	\subsubsection{Quadratic objective function} \label{subsec:quad_obj}
	
	\begin{figure}
		\centering
		{\includegraphics[width=0.8\textwidth]{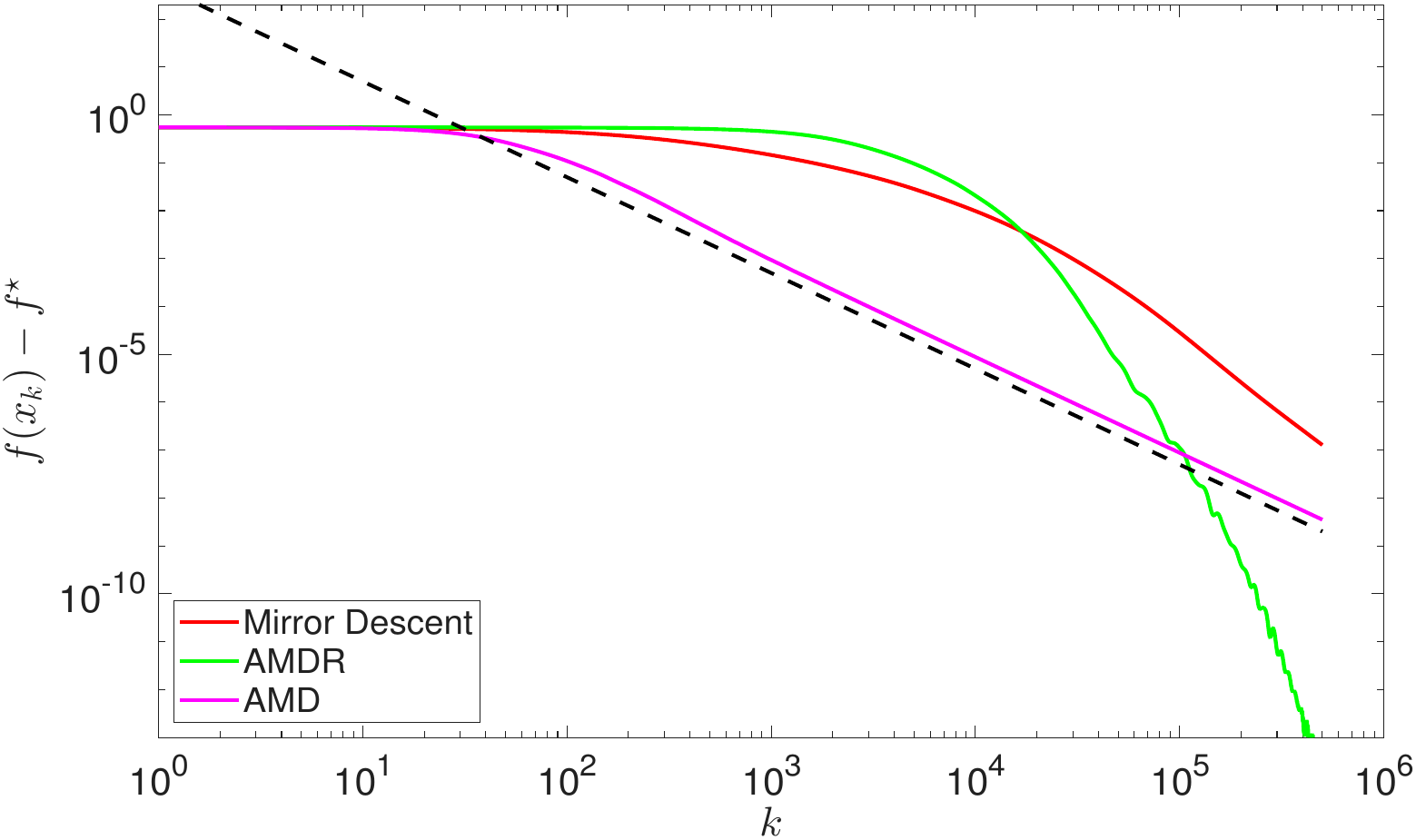}}
		\caption{\textcolor{red}{Quadratic objective function, $f(x_k)-f(x^\star)$ vs.\ $k$. The dotted line has a slope corresponding to a decay $1/k^2$.}}
		\label{fig:secondexperiment}
	\end{figure}
	Now, {\color{black} following \cite{KBB15},} the objective function is \(f(x) = (1/2) x^T B^T B x\), with \(B\) a \(d\times d\) matrix with entries  given by independent standard normal random variables. The initial $x_0$ is chosen randomly by generating a vector with independent, uniformly random components in \([0,1]\) and then rescaling to ensure that \(\sum_j x_j = 1\).  The smoothness constant \(L_f\) for the gradient $\nabla f(x) = B^TBx$ is the norm of $B^TB$ as an operator from \((\mathbb{R}^d, \ell^1)\) to \((\mathbb{R}^d, \ell^\infty)\), which is given by the maximum \(m(B^TB)\) of the absolute value of the entries. In Assumption \ref{as:cond_mir1}, \(L_\chi =1\) and, in view of condition \eqref{eq:h} in Theorem~\ref{th:decay}, we run  AMD with a learning rate \(h=1/m(B^TB)\); the same value is used for mirror descent. For AMDR we follow the prescription in \cite{KBB15} and set \[h = \sqrt{\epsilon/\big(2(1+d\epsilon)m(B^TB) \gamma\big)}.\] The experiment in Figure~\ref{fig:secondexperiment} has \(d=1000\) and \textcolor{red}{500,000} steps.
	The minimizer is not in the relative interior (in fact has 323 vanishing components) and the results in Section~\ref{subsec:alternative} are necessary to establish the convergence of AMDR and AMD. In addition, in the experiment, \textcolor{red}{\(m(B^TB) \approx 1.1\times 10^{3}\)} which leads to learning rates \textcolor{red}{\(h \approx 2.5\times 10^{-4}\)} for mirror descent and AMD and \textcolor{red}{\(h\approx 6.6\times 10^{-4}\)} for AMDR. The figure clearly bears out the \(1/k^2\) acceleration proved in Theorem~\ref{th:decaybis} for AMD. Mirror descent and AMDR lead initially to very little decay in \(f\) but they decay faster than \(1/k^2\) once they are near the minimizer.
	In this particular experiment they are both outperformed by AMD \textcolor{red}{for the first $10^5$ iterations}.
	
	\subsubsection{Quadratic objective function, larger learning rates}
	\begin{figure}
		\centering
		{\includegraphics[width=0.8\textwidth]{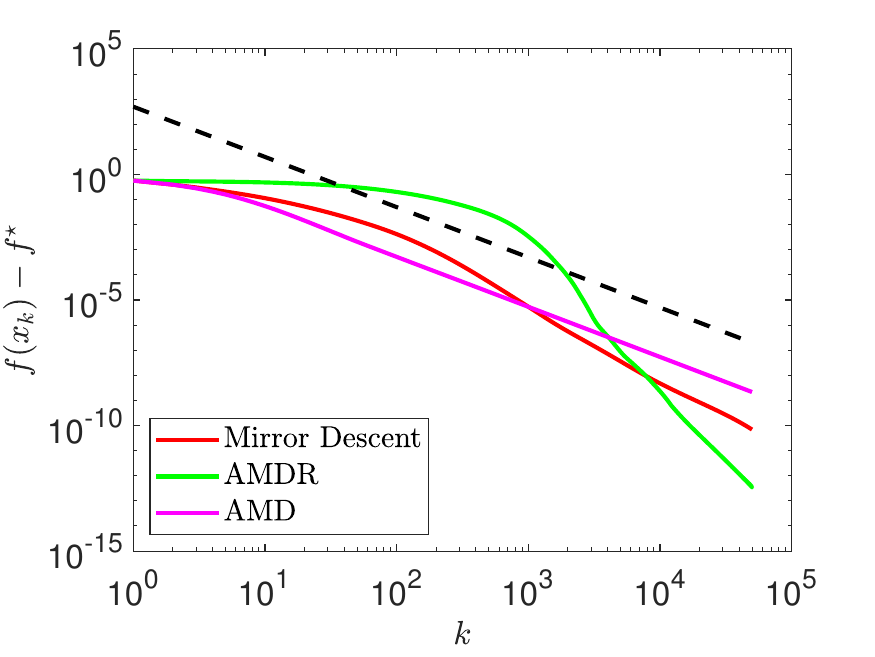}}
		\caption{Quadratic objective function, larger learning rates, $f(x_k)-f(x^\star)$ vs.\ $k$. The dotted line has the same equation as the reference line in Figure~\ref{fig:secondexperiment} so as to make it easy to compare both figures.}
		\label{fig:thirdexperiment}
	\end{figure}
	
	Numerical experimentation reveals that the recipes we have just used to determine the learning rates are too pessimistic; the three algorithms tested may operate with substantially larger values of \(h\), thus providing a larger decay in \(f\) for a given number of iterations. In fact, the values of \(h\) were based on the size of \(m(B^TB)\) (an operator norm for the matrix in \(\nabla f(x)=B^TBx\)). However, in \eqref{eq:edosimplex} and  in \eqref{eq:edosimplexbis} below we see that in the differential equations being approximated by the algorithms, \(\nabla f(x)\) is premultiplied by \(D(z)={\rm diag}(z)\) or \(D(x)={\rm diag}(x)\) respectively. Once the solution is close to the minimizer, those matrices are close to \(D(x^\star)\) and it is reasonable to think that  {\color{black} near the minimizer} the learning rates should really be determined by the size of the matrix \(D(x^\star)B^TB\)  rather than by the size of \(B^TB\). If \(D(x^\star)B^TB\) is much smaller than \(B^TB\) learning rates based on the size of \(B^TB\) may be expected to be unduly pessimistic.	
		
	These considerations may be related to the notion of \emph{relative smoothness} introduced in \cite{LFN18} (see also \cite{BBT17}), an alternative to the notion of (absolute) smoothness in \eqref{eq:smooth}. For the sake of brevity, we only present the concept of relative smoothness as it applies to the simplex.
	A real function \(g\), twice continuously differentiable, is said to be \(L_r\)-relatively smooth with respect to the negative entropy \(\phi\) in \eqref{eq:negent}, if for \(z\) in the relative interior of the simplex
	\begin{equation} \label{eq:rsmooth}
	\nabla^2 g(z) \preceq L_r \nabla^2 \phi(z).
	\end{equation}
	{\color{black} Note that in the original definition of relative smoothness the function $g$ does not need to be twice differentiable, but, if it is, then the definition coincides with \eqref{eq:rsmooth}. }
	Now, recalling that \(D(z)\) is the inverse of \(\nabla^2 \phi(z)\), we may equivalently write
	\[D(z)^{1/2} \nabla^2 g(z) D(z)^{1/2} \preceq L_r I,
	\]
	and therefore the best possible \(L_r\) is given by the maximum of the spectral radius of the symmetric matrix \(D(z)^{1/2} \nabla^2 g(z) D(z)^{1/2}\) as \(z\) ranges in the relative interior. Note that this symmetric matrix is similar to the matrix \(D(z)\nabla^2 g(z)\) and therefore shares its eigenvalues. For \(f(z) = (1/2) z^T B^T B z\), \(D(z)\nabla^2 f(z) = D(z)B^TB\) and therefore \(L_r\) is the maximum eigenvalue of
	\(D(z)B^TB\) (while as pointed out above \(L_f\) is the maximum of the entries of \(B^TB\)).
	It may then be conjectured that the assumption of relative smoothness of \(f\) as in \eqref{eq:smooth} could be replaced by the assumption that the objective function be \(L_r\)-relatively smooth and that the algorithms may be operated with learning rates based on using the value of \(L_r\) rather than on the value of \(L_f\).
	
	To investigate this conjecture, we revisit the experiment in Figure~\ref{fig:secondexperiment}. There, as mentioned before,
	\(x^\star\) has 323 vanishing entries. In addition the maximum entry of \(x^\star\) happens to be \(5.9\times10^{-3}\) (since the entries add up to 1 they might be expected to be small). Thus, the entries of \(D(x^\star)B^TB\)  are more than two orders of magnitude smaller than those \(B^TB\). We introduced the symmetric matrix
	\[M = D(x^\star)^{1/2}B^TBD(x^\star)^{1/2}\]
	and estimated {\color{black} roughly} \(L_r\) by the spectral radius \(\rho(M)\).\footnote{Of course, this would not make sense in a real application, as \(M\) requires the knowledge of \(x^\star\). In practice \(L_r\) could be estimated as the spectral radius of
		\( D(x_k)^{1/2}B^TBD(x_k)^{1/2}\) for a suitable {\color{black} range of values of \(k\)}. }
	Then we used the learning rate \(h=1/(L_\chi L_r)\) (rather than \(1/(L_\chi L_f)\)) for mirror descent and AMD and \[h = \sqrt{\epsilon/\big(2(1+d\epsilon)\rho(M)\big) \gamma}\] for AMDR. Figure~\ref{fig:thirdexperiment} has the same realizations of the random elements \(B^TB\) and  \(x_0\) as Figure~\ref{fig:secondexperiment}, the only difference being that we now use the \(L_r\)-based values of the learning rates just described; these turn out to be
	\(h \approx 1.3\times 10^{-1}\) for mirror descent and AMD and \(h\approx 8.1\times 10^{-3}\) for AMDR. In Figure~\ref{fig:thirdexperiment} each optimization method qualitatively behaves very much as it did in Figure~\ref{fig:secondexperiment}; however for each method the size of \(f(x_k)-f(x^\star)\) for given  \(k\) is now clearly smaller than it was.
	
	This experiment shows the interest of future analyses of AMDR and AMD  replacing the notion of absolute smoothness by the notion of relative smoothness, similarly to what it is done {\color{black} with mirror descent} in \cite[Theorem 3.1]{LFN18}.
	
	{\color{black}
	\subsection{Convex Learning Problem constrained to an $\ell_p$ ball}
	
	We observe a dataset $\{(x_i,y_i)\}_{i=1}^n$ with feature vectors
	$x_i \in \bbE^\star$ and binary responses $y_i \in \{-1,+1\}$.
	Throughout, we model the conditional distribution of $y_i$ given $x_i$
	via a standard logistic model,
	\[
	\mathbb{P}(y_i = 1 \mid x_i)
	= \sigma(\langle w^\star, x_i\rangle),
	\qquad
	\sigma(z) = \frac{1}{1+\exp(-z)},
	\]
	so that $y_i$ reflects a noisy sign of the underlying linear score
	$\langle w^\star, x_i\rangle$ for some fixed but unknown vector $w^\star \in \bbE$.
	We consider linear predictors $g_w(x)=\langle w,x\rangle$ with
	parameters $w \in \mathbb{E}$, and we aim to determine $w$ that
	minimises the empirical risk
	\[
	f(w)
	= \frac{1}{n}\sum_{i=1}^n
	\varphi(\langle w,x_i\rangle, y_i),
	\]
	where $\varphi$ is the logistic loss $\varphi(t,y)=\log(1+\exp(-y t))$.
	
	We impose that the predictor $g_w$ is $r$--Lipschitz with respect to the
	$\|\cdot\|_q$ norm for some $q\ge2$, which is equivalent to requiring
	$\|w\|_p \le r$ with $p = q/(q-1)$.
	Thus $w$ is constrained to the $\ell_p$ ball of radius $r$, and we
	endow $\bbE$ with the $\|\cdot\|_p$ norm, defining the feasible set
	\[
	\mathcal{X}
	= \{\, w \in \bbE : \|w\|_p \le r \,\}.
	\]
	
	{\color{black} This constraint is equivalent to using $\|w\|_p^p$ as a penalization term which is known as bridge regression for general $p$, and was first considered by \cite{frank1993statistical}, an algorithm based on Newton's method to solve the problem  proposed in \cite{Fu1998397}, with asymptotics studied in \cite{knight2000asymptotics}. }
	
	Note that if $t\mapsto \varphi(t,y)$ is continuously twice differentiable with gradient Lipschitz constant $L_\varphi$ uniformly in $y$ then
	\begin{equation*}
		L_f = \frac{L_\varphi}{n} \lVert X^TX\rVert_{p,q}
	\end{equation*}
	where $X = (x_1,\ldots,x_n)$ and $\lVert \cdot\rVert_{p,q}$ is the induced matrix norm from $\ell_p$ to $\ell_q$.
	
	We will use the mirror map {\color{black}$\chi= \nabla \psi^\star$, $\psi=\phi+\delta_{\cal X}$} corresponding to
	\begin{equation*}
		\phi(w) = -\log(r^p-\lVert w\rVert_p^p).
	\end{equation*}
	Note that $\phi$ is $\mu_\phi$-strongly convex with respect to the $\lVert \cdot \rVert_p$ norm and {\color{black}
	$\mu_\phi = (p-1)=/{r^p}$.}
	Then we have $L_\chi = 1/\mu_\phi$.

For each $\zeta \in \bbE^\star$ we calculate $\chi(\zeta)$ by solving the equation $\nabla \phi(w)= \zeta$, that is {\color{black} by solving} the system
	\begin{equation*}
		\zeta_i = \partial_{w_i} \phi(w) = \frac{p|w_i|^{p-1}\mathrm{sign}(w_i)}{r^p-\lVert w\rVert_p^p}, \quad i=1,\ldots,n.
	\end{equation*}
	Observe that this implies $\mathrm{sign}(w_i) = \mathrm{sign}(\zeta_i)$ and
	\begin{equation*}
		|\zeta_i| =  \frac{p|w_i|^{p-1}}{r^p-\lVert w\rVert_p^p}, \quad i=1,\ldots,n.
	\end{equation*}
	Taking the $q$-th power and summing over $i$ gives
	\begin{equation*}
		\lVert \zeta\rVert_q^q =  \frac{p^q\lVert w\rVert_p^p}{(r^p-\lVert w\rVert_p^p)^q}.
	\end{equation*}
	Therefore the solution $\chi(\zeta)$ is given by
	\begin{align*}
		\chi(\zeta)_i = \left(\frac{u|\zeta_i|}{p}\right)^{\frac{1}{p-1}} \mathrm{sign}(\zeta_i),
	\end{align*}
	where $u=r^p-\lVert w\rVert_p^p$ is the solution of the scalar equation $G(u)=0$ with $u\in [0,r^p]$ and
	\begin{equation*}
		G(u) = \frac{\lVert \zeta\rVert_q^q}{p^q} u^q -r^p+u.
	\end{equation*}
	Observe that $G$ is strictly increasing on $[0,r^p]$, $G(0)<0$ and $G(r^p)>0$ therefore there is a unique solution. While there is not an explicit solution to $G(u)=0$, as this is a scalar equation it can be solved efficiently.
	
	For AMDR we also need to select a regularizer $R$.  As suggested in \cite{KBB15} we {\color{black} first consider choices where $R(u,v) = D_g(u,v)$ is the Bregman divergence} induced by a function $g$ which is $\ell_R$ strongly convex and $L_R$-smooth {\color{black} as in \eqref{eq:hypothesesonR}}. The choice $g=\psi$ is not suitable since $\psi$ is not smooth. A valid choice would be to set $g=\frac{1}{2}\lVert \cdot\rVert_2$, {\color{black} so that the Bregman divergence is the Euclidean distance; unfortunately} this leads to $\ell_R = d^{-\frac{1}{p}+\frac{1}{2}}$ {\color{black} and therefore} for large dimension the step size needs to be very small {\color{black} which implies poor pereformance}. Alternatively, we could {\color{black} think of using} a mirror map which incorporates the geometry of $\bbE$ such as $g_2(w) = ||w||^2$ or $g_p(w)=||w||^p$; however $g_2$ is not strongly convex and $g_p$ is not smooth.
	
	On the other hand it is possible to use $R(x,y) = \frac{1}{2} \lVert x-y\rVert_p^2$ which satisfies the convexity and smoothness conditions with $\ell_R=L_R=1$. With this choice of regularizer the final step of AMDR is given by
	$$x_{k+1} \gets {\rm arg min}_{x\in \cal X}\big(\gamma h \langle \nabla f(y_k),x \rangle+\frac{1}{2} \lVert x-y_k\rVert_p^2 \big)$$
	which is computationally demanding to solve because the objective involves the square of the $\lVert \cdot \rVert_p$-norm, which is non‑quadratic and non‑separable when $p\neq 2$. Its gradient is continuous but not Lipschitz, and the Hessian contains terms of the form $|x_i-y_i|^{p-2}$, which become unbounded for \(x_i-y_i\) near zero when $p<2$. At the same time, the constraint $\lVert x\rVert_p\leq r$ introduces a nonlinear boundary whose Karush-Kuhn-Tucker conditions cannot be solved in closed form unless $p=2$. As a result, each evaluation of the optimality conditions requires solving a non-linear, coupled system where neither the objective nor the constraint decouples across coordinates. Therefore, in this case we have chosen not to include AMDR as experiments proved it to be too expensive per iteration to be competitive.
	
	For our experiment we generate independent observations $\{(x_i,y_i)\}_{i=1}^n$ as follows. Each feature vector $x_i \in \mathbb{R}^d$ is drawn from a standard Gaussian distribution, $x_i \sim \mathcal{N}(0,I_d)$. We sample a ground--truth parameter vector $w^\star \in \mathbb{R}^d$ independently from the same distribution. The initial $w_0$ is set to be the zero vector, while $d=100$, $n=500$ and $p=1.5$. We run both MD and AMD with $h=(L_\chi L_f)^{-1}$. In Figure~\ref{fig:convexlearningp15} we plot the convergence of the objective to the minimizer which is calculated using a long run of projected gradient descent. We see in this case MD converges with the rate $1/k$ while AMD converges with the accelerated rate $1/k^2$.
	\begin{figure}
		\centering
		\includegraphics[width=\linewidth]{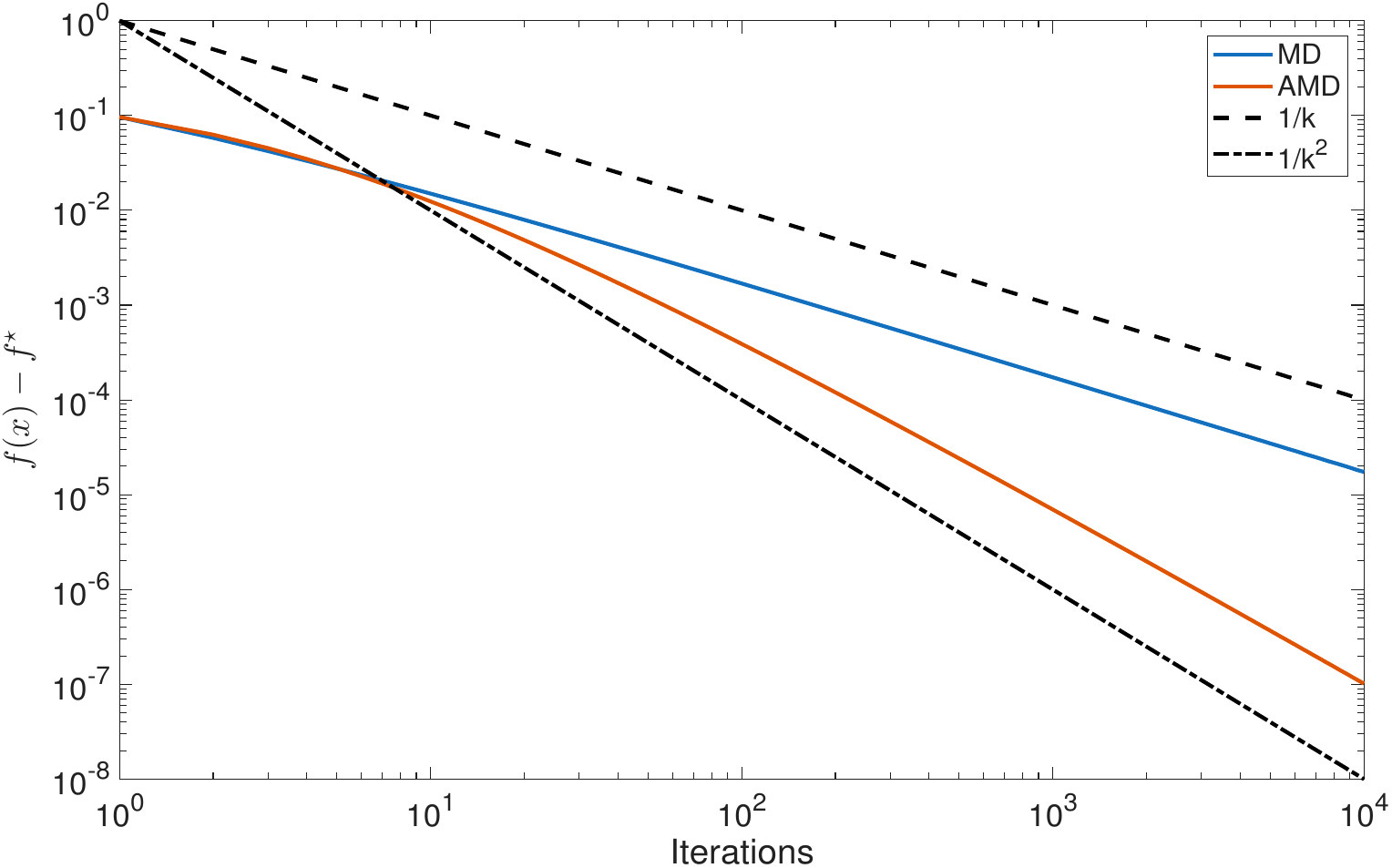}
		\caption{Convex learning problem constrained to an $\ell_p$ ball, $f(X_k)-f(x^*)$ vs. $k$. The dotted line has a slope corresponding to a decay $1/k$ and $1/k^2$}
		\label{fig:convexlearningp15}
	\end{figure}
}


\section*{Funding and/or Conflicts of interests}
The authors have no competing interests to declare that are relevant to the content of this article. JMS has been funded by Ministerio de Ciencia e Innovaci\'{o}n (Spain), project PID2022-136585NB-C21, MCIN/AEI/10.13039/501100011033/FEDER, UE. PD and KCZ acknowledge support from the EPSRC grant EP/\-V006177/1.
	
\vskip 0.2in

\section*{Appendix: A primal accelerated ODE and its discretizations}
	\label{subsec:primal_acc_ODE}
	In the literature several ODEs and algorithms have appeared that have similarities to those in \cite{KBB15}. {\color{black} For the sake of completeness, we shall now make explicit these similarities}
in the case of  ODEs appearing in \citep{WWJ16,WRJ21} and one of the algorithms given in \cite{T08}. One notable difference is that, contrary to the setting in \cite{KBB15}, only primal variables are used in those references. Motivated by this, we will now give formulations of the ODE \eqref{eq:Kr_ode} and Algorithms \ref{alg:amd_kr} and \ref{alg:amd_us} that only make use of primal variables. The developments parallel those in Section
	\ref{ss:rewritingprimal}. In this Appendix it is assumed that Assumption \ref{ass:two} holds.
	
	\subsection*{\color{black}Primal writing in the case \({\cal N} = \{0\}\)}
	\paragraph{The differential system}
	
	In the case where $\mathcal{N}$ is trivial, the one-to-one correspondence between \(z=\chi(\zeta)\) and \(\zeta=\nabla \phi(z)\)  may be used to rewrite the system \eqref{eq:Kr_ode} in terms of the primal variable \(z\) as follows:
	\begin{subequations}\label{eq:wode}
		\begin{eqnarray}
			\frac{d}{dt}\nabla \phi(z(t))  &=& -\frac{t}{r} \nabla f(x(t)),\label{eq:wodez}\\
			\dot{x}(t) &=& \frac{r}{t} \big(z(t)-x(t)\big).\label{eq:wodex}
		\end{eqnarray}
	\end{subequations}
	For the value \(r=2\), this system is a particular instance of the systems derived in \cite{WRJ21}  through  Lagrangian functions. Note that when \(\cal N\) is not reduced to \(\{0\}\), the differential equation \eqref{eq:wodez} is meaningless because the left hand-side is constrained to be a vector tangent to the manifold \(\cal M\) and \(\nabla f(x)\) is not constrained in that way (see the right panel in Figure~\ref{fig:simplex_diagram}).
	
	\paragraph{The algorithms}
	
	As is the case for the differential system, when \({\cal N}= \{0\}\), AMDR and AMD may be expressed avoiding dual variables. Restricting the attention to AMD (AMDR may be dealt with analogously), we have
	\begin{eqnarray}
		y_k &=& x_k +\frac{1}{\gamma_k} \big(z_k-x_k),\nonumber\\
		\nabla \phi(z_{k+1})&=& \nabla \phi(z_k) - \gamma_k h\nabla f(y_k),\label{eq:wz}\\
		x_{k+1} &=& y_k +\frac{1}{\gamma_k}\big(z_{k+1}-z_k\big).\nonumber
	\end{eqnarray}
	Formulas similar to \eqref{eq:wz} appear in \citep{WWJ16,WRJ21}. Once \(\nabla \phi(z_{k+1})\) has been found via \eqref{eq:wz}, \(z_{k+1}\) is recovered as
	\(\chi(\nabla \phi(z_{k+1}))\). Note that when  \(\cal N\) is not reduced to \(\{0\}\), \eqref{eq:wz} is meaningless because nothing guarantees that its right hand-side lies in \(\cal M\),
	the set where \(\nabla \phi\) takes values.
	
	\subsection*{Primal writing in the general case}
	\paragraph{The differential system}
	
	It is possible  to obtain a system of differential equations  satisfied by the primal variables \(z(t)=\chi(\zeta(t))\) and \(x(t)\) without the hypothesis \({\cal N}= \{0\}\), where \(\zeta(t)\) and \(x(t)\) are solutions of \eqref{eq:Kr_ode}. In fact,
	by arguing as in Section~\ref{ss:rewritingprimal} and denoting by \(z\) the mirror of the variable \(\zeta\), one may write:
	\begin{subequations}\label{eq:odemine}
		\begin{eqnarray}
			\dot{z}  &=& \chi^\prime(\nabla \phi(z(t)))\left(-\frac{t}{r} \nabla f(x(t))\right),\label{eq:odezmine}\\
			\dot{x} &=& \frac{r}{t} \big(z(t)-x(t)\big);\label{eq:odexmine}
		\end{eqnarray}
	\end{subequations}
	{\color{black} This ODE has appeared before in the literature in \cite{KBB16}}.
	In the particular case where \({\cal N}= \{0\}\), this reduces to \eqref{eq:wode}.
	The system \eqref{eq:odemine} has to be initialized with \(x(0) =z(0)\) in the relative interior of \(\cal X\), and
	from \eqref{eq:Vhat}, it has the Lyapunov function:
	\begin{equation}\label{eq:lyapprimal}
		\frac{t^2}{r^2} \big(f(x)-f(x^\star)\big)+ D_\phi\big(x^\star,z).
	\end{equation}
	
	\begin{example} [Example~\ref{ex:simplex} continued]
		In the case of the simplex, \eqref{eq:odezmine} reads
		\begin{equation}\label{eq:edosimplexbis}
			\dot{z} = D(z) \Big( -\frac{t}{r} \nabla f(x)+ \langle -\frac{t}{r} \nabla f(x) ,z \rangle {\bf 1}\Big),
		\end{equation}
		where \(D(z)\) is the diagonal matrix with entries \(z_i\); \end{example}

	\paragraph{The algorithms}
	It is also possible, without any assumption on the dimension of \(\cal N\), to express AMDR and AMD in terms of the primal points \(z_k=\chi(\zeta_k)\). For brevity we only give details for AMD, which, by arguing as in Section~\ref{ss:rewritingprimal}, may be reformulated as
	\begin{subequations}\label{eq:ALG}
		\begin{eqnarray}
			y_k &=& x_k +\frac{1}{\gamma_k} \big(z_k-x_k),\label{eq:ymine}\\
			z_{k+1}&=& \chi\big(\nabla \phi(z_k) - \gamma_k h\nabla f(y_k)\big),\label{eq:zmine}\\
			x_{k+1} &=& y_k +\frac{1}{\gamma_k}\big(z_{k+1}-z_k\big).\label{eq:xmine}
		\end{eqnarray}
	\end{subequations}
	This is  a particular case of the algorithms studied in \cite{T08}. 
	Note that in \cite{T08} there is no discussion of the continuous time limit; however it is not difficult to show that, when \eqref{eq:asympgamma} is fulfilled, the algorithm \eqref{eq:ALG} is a consistent  discretization of the system \eqref{eq:odemine}.
	Furthermore, in terms of primal variables, the discrete-time Lyapunov function in \eqref{eq:lyapbis} is given by
	\[ (\gamma_k^2-\gamma_k)h \big(f(x_k)-f(x^\star)\big)+ D_\phi\big(x^\star,z_k\big).
	\]
	Under the consistency requirement \eqref{eq:asympgamma},
	this is an approximation to the continuous-time Lyapunov function \eqref{eq:lyapprimal}.

\end{document}